\documentclass{amsart}

\usepackage{amssymb, amsmath}

\newtheorem{theorem}{Theorem}
\newtheorem*{theorem*}{Theorem}
\newtheorem{proposition}[theorem]{Proposition}

\newtheorem*{fact*}{Fact}
\newtheorem*{factA}{Fact A}
\newtheorem*{factB}{Fact B}
\newtheorem{definition}[theorem]{Definition}
\newtheorem*{definition*}{Definition}

\newtheorem{remark}[theorem]{Remark}

\numberwithin{theorem}{section}
\numberwithin{equation}{section}

\usepackage{epsfig}
\usepackage{color}
\usepackage{graphicx}
\usepackage{wrapfig}
\usepackage{mathrsfs}
\usepackage{pinlabel}

\definecolor{medgreen}{rgb}{0 , 0.7 , 0}
\definecolor{purple}{rgb}{0.7 , 0 , 0.7}

\newcommand{\note}[1]{
}

\author{Fran\c{c}ois Gu\'eritaud}
\address{CNRS and Universit\'e Lille 1, Laboratoire Paul Painlev\'e, 59655 Villeneuve d'Ascq Cedex, France 
\newline Wolfgang-Pauli Institute, University of Vienna, CNRS-UMI 2842, Austria}
\email{francois.gueritaud@math.univ-lille1.fr}
\thanks{Partially supported by the Agence Nationale de la Recherche under grants ETTT (ANR-09-BLAN-0116-01), DiscGroup (ANR-11-BS01-013), Labex CEMPI (ANR-11-LABX-0007-01).}
\title[Veering triangulations and Cannon-Thurston maps]{Veering triangulations \\ and Cannon-Thurston maps}
\date{June 2015}

\begin{document}
\begin{abstract}
Any hyperbolic surface bundle over the circle gives rise to a continuous surjection from the circle to the sphere, by work of Cannon and Thurston. We prove that the order in which this surjection fills out the sphere is dictated by a natural triangulation of the surface bundle (introduced by Agol) when all singularities of the invariant foliations are at punctures of the fiber.
\end{abstract}

\maketitle

\section{Introduction}

Surface bundles over the circle are historically an important source of examples in hyperbolic 3-manifold theory. Thurston proved that, barring natural topological obstructions, they always carry complete hyperbolic metrics, which was a first step towards Perelman's geometrization of 3-manifolds. Cannon and Thurston \cite{CT} found surprising sphere-filling curves naturally associated to hyperbolic surface bundles.

In \cite{Agol}, Agol singled out a special class of hyperbolic surface bundles: the ones for which singularities of the invariant foliations occur only at punctures of the fiber.
He proved that such surface bundles come with a natural (topological) ideal triangulation.

The purpose of this paper is to exhibit, for such surface bundles, a correspondence between Agol's triangulation and the corresponding Cannon-Thurston map. The correspondence takes the form of a pair of tessellations of the plane $\mathbb{C}$: (1) the link of a vertex $\Omega$ of the universal cover of Agol's triangulation; (2) a plane tiling recording the order in which the Cannon-Thurston map fills out the sphere $\mathbb{C}\cup\{\infty\}$, switching colors at each passage through the parabolic fixed point $\infty$. Object (1) is clearly a triangulation of the plane, though it may not be realized by non-overlapping Euclidean triangles in $\mathbb{C}$ (Agol's triangulation is only topological, not geodesic). Object (2) is clearly a partition of $\mathbb{C}$, although it will take work to determine that it is actually a tessellation into topological disks (typically with fractal-looking boundaries). In the end (Theorem~\ref{thm:dico} below), the two tessellations turn out to fully determine each other at the combinatorial level, and in particular share the same vertex set. This connection between tessellations was previously known for punctured torus bundles by results of Cannon--Dicks \cite{CD} and Dicks--Sakuma \cite{Dicks-Sakuma}, whose work was a crucial inspiration. These results were announced in \cite{annonce}.

\subsection{Hyperbolic mapping tori and invariant foliations}
Let $S$ be an oriented surface with at least one puncture, and $\varphi:S\rightarrow S$ an orientation-preserving homeomorphism. Define the \emph{mapping torus} $M_\varphi:=S\times [0,1]/\!\sim_\varphi$, where $\sim_\varphi$ identifies $(x,1)$ with $(\varphi(x),0)$. The topological type of the 3-manifold $M_\varphi$ depends only on the isotopy type of $\varphi$. 

Suppose $S$ has a half-translation structure, \emph{i.e.}\ a singular Euclidean metric with a finite number of conical singularities of cone angle $k\pi$ $(k\geq 3)$, and total cone angle $k'\pi$ ($k'\geq 1$) around each puncture. Every straight line segment in $S$ then belongs to a unique (singular) foliation by parallel straight lines. The surface $S$ with cone points removed admits an isometric atlas over $\mathbb{R}^2$ whose chart maps are all of the form $(x,y)\mapsto (\alpha,\beta)\pm (x,y)$ for some reals $\alpha,\beta$. The group $\mathrm{SL}_2(\mathbb{R})$ acts on the space of such atlases by composition with the charts, hence $\mathrm{PSL}_2(\mathbb{R})$ acts on the space of (isometry classes of) half-translation surfaces endowed with a privileged pair of perpendicular foliations by straight lines. A landmark result of Thurston's \cite{FLP, Otal} is 
\begin{factA} \label{fact:bundles}
Suppose the isotopy class $[\varphi]$ preserves no finite system of simple closed curves on $S$ ($\varphi$ is called \emph{pseudo-Anosov}). Then there exists a half-translation structure on $S$ such that $\varphi:S\rightarrow S$ is realized by a diagonal element 
$(\begin{smallmatrix} \alpha & \\ & \alpha^{-1} \end{smallmatrix})$, where $\alpha>1$; the vertical and horizontal foliations of $S$ for this structure, called $\lambda^+$ and $\lambda^-$, are preserved by $\varphi$ and come equipped with transverse measures that are preserved up to a factor $\alpha$ (resp. $\alpha^{-1}$). 

Moreover, the mapping torus $M_\varphi$ admits a (unique) complete hyperbolic metric: $M_\varphi\simeq \Gamma\backslash \mathbb{H}^3$ for some discrete group $\Gamma < \mathrm{PSL}_2\mathbb{C}=\mathrm{Isom}_0(\mathbb{H}^3)$.
\end{factA}
In Thurston's proof, which gave the first abundant source of examples of hyperbolic 3-manifolds, the foliations $\lambda^{\pm}$ are in fact an important tool to construct the hyperbolic metric on $M_\varphi$. The half-translation structure on $S$ in the result above is itself unique up to the action of diagonal elements of $\mathrm{PSL}_2(\mathbb{R})$.

\subsection{Combinatorics of veering triangulations}
We now describe a construction of Agol's triangulation, an alternative to \cite{Agol} which may be of separate interest.
An ideal tetrahedron is a space diffeomorphic to a compact tetrahedron minus its 4 vertices. An ideal triangulation of a $3$-manifold $M$ is a realization of $M$ as a union of finitely many ideal tetrahedra, glued homeomorphically face-to-face.

\begin{definition*}
A \emph{taut structure} on an ideal triangulation of an oriented $3$-manifold $M$ (into $n$ tetrahedra) is a map from the set of all $6n$ dihedral angles of the tetrahedra into $\{0,\pi\}$ such that each tetrahedron has one pair of opposite edges labeled $\pi$ and all other edges labeled $0$; and each degree-$k$ edge of $M$ is adjacent to precisely two angles labelled $\pi$ and $(k-2)$ edges labelled $0$.
\end{definition*}
A taut structure can be viewed as a crude attempt at endowing the tetrahedra of $M$ with geometric shapes in order to realize the hyperbolic metric (very crude indeed since all tetrahedra look flat!).

In a rhombus of $\mathbb{R}^2$ symmetric across both coordinate axes, we call the two edges with positive slope \emph{rising}, the other two edges \emph{falling}. The diagonals, which are segments of the coordinate axes, are called vertical and horizontal.
\begin{definition*} \label{def:veering}
A taut structure on an ideal triangulation of an oriented $3$-manifold $M$ is called \emph{veering} if its edges can be $2$-colored, in red and blue, so that every tetrahedron can be sent by an orientation-preserving map to the one pictured in Figure~\ref{tetrahedron}: a thickened rhombus in $\mathbb{R}^2\times \mathbb{R}$ with $\pi$ on the diagonals and $0$ on other edges; with the vertical diagonal in front, the horizontal diagonal in the back, rising edges red, and falling edges blue. (The diagonals might be any color.) 
\end{definition*}
\begin{figure}[h!]
\centering
\labellist
\small\hair 2pt
\pinlabel $\pi$ [c] at 44 50
\pinlabel $\pi$ [c] at 30 39
\pinlabel $0$ [c] at 20 55
\pinlabel $0$ [c] at 20 15
\pinlabel $0$ [c] at 73 55
\pinlabel $0$ [c] at 73 15
\pinlabel $\pi$ [c] at 137 45
\pinlabel $0$ [c] at 122 30
\pinlabel $0$ [c] at 153 30
\pinlabel {base} [c] at 137 19
\pinlabel {tip} [c] at 138 57
\endlabellist
\includegraphics[width=9cm]{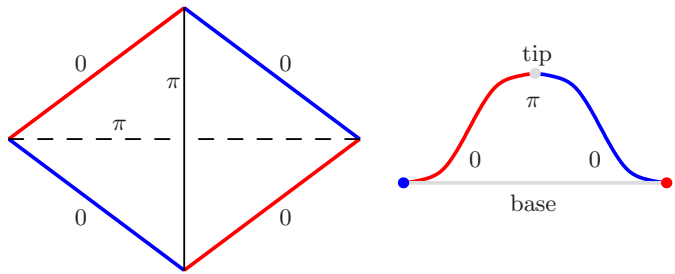}
\caption{\emph{Left}: a flattened tetrahedron with 4 colored edges. 
\newline 
\emph{Right}: the triangular link at any of the 4 vertices. Angles 0 and $\pi$ are indicated by a graphical, train-track-like convention. The tip and base, drawn in grey, receive colors (blue/red) from the adjacent triangles. The triangle is called \emph{hinge} if and only if the tip and base have different colors.}
\label{tetrahedron}
\end{figure} 
In \cite{HRST} and \cite{FG}, veering triangulations are shown to admit \emph{positive angle structures}: this is a less crude (linearized) version of the problem of finding the complete hyperbolic metric on $M$ endowed with a geodesic triangulation. Interestingly however, Hodgson, Issa and Segerman found veering triangulations that are in fact not realized geodesically but have instead some tetrahedra turning ``inside out'' \cite{Issa}.

In \cite{Agol}, Agol described a \emph{canonical, veering} triangulation of a general hyperbolic mapping torus $M=M_\varphi$, provided all singularities of the foliations $\lambda^+, \lambda^-$ occur at punctures of the fiber~$S$. 
Our first main result is an alternative construction of Agol's triangulation (details in Section~\ref{sec:triangulation}).
\begin{theorem} \label{thm:veering}
Suppose all singularities of the invariant foliations $\lambda^\pm$ of the pseudo-Anosov monodromy $\varphi:S\rightarrow S$ are at punctures of $S$. 
Any maximal immersed rectangle $R$ in $S$ with edges along leaf segments of $\lambda^\pm$ contains one singularity in each of its four sides. Connecting these four ideal points, and thickening in the direction transverse to $S$, yields a tetrahedron $\Delta_R\subset S\times \mathbb{R}$. The tetrahedra $\Delta_R$ glue up to yield a veering triangulation of $S\times \mathbb{R}$, compatible with the equivalence relation $(x,t+1)\sim_\varphi (\varphi(x),0)$, hence descending to $M_\varphi$.
\end{theorem}

The veering structure in the above theorem is given as follows: $\Delta_R$ has its $\pi$-angles at the edges connecting points belonging to opposite sides of $R$ (the edge connecting horizontal sides being in front); an edge of positive slope is red; an edge of negative slope is blue.

Remark: unlike Agol's original definition, this construction does not rely on any auxiliary choices (\emph{e.g.}\ of train tracks): as a result it should generalize to the Cannon-Thurston maps of degenerate surface groups built by Mahan Mj \cite{Mj}, when the ending laminations define foliations with no saddle-point connection. The focus of this paper being combinatorics, we choose however to restrict to surface bundles.

\medskip
In order to state the main result, we now point out some features inherent to any veering triangulation $\mathcal{T}$ (see Figure~\ref{link}, and \cite{FG} for detailed proofs). The link of a vertex $\Omega$ of the universal cover $\widetilde{ \mathcal{T}}$ is a tessellation $\Delta$ of the plane. The vertices and edges of $\Delta$ receive colors (red/blue) from the edges of $\mathcal{T}$.

\begin{figure}[h!]
\centering
\labellist
\small\hair 2pt
\endlabellist
\includegraphics[width=8cm]{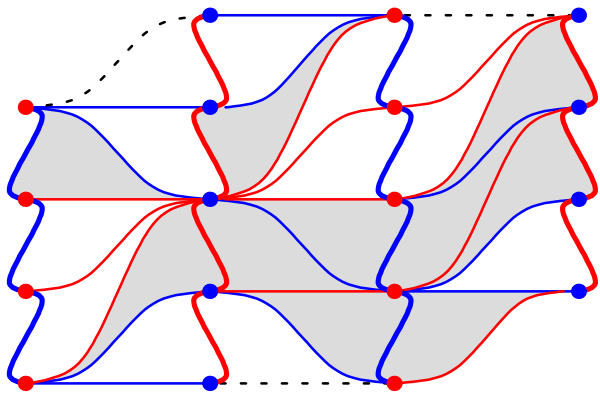}
\caption{Three adjacent ladders in the vertex link $\Delta$ of a veering triangulation. Ladderpoles (vertical) are slightly thicker. The middle ladder is ascending (tips above base rungs). The colors of the rungs are determined by the combinatorics of this chunk of $\Delta$, except the 3 dotted rungs. Triangles with no dotted edge are shaded when they are hinge.}
\label{link}
\end{figure} 

Among the triangles of $\Delta$, we may distinguish two types: a triangle coming from truncating a tetrahedron (thickened rhombus) whose diagonals are of opposite colors is called \emph{hinge}; other triangles are \emph{non hinge}. See Figure~\ref{tetrahedron} and its caption.

An edge of $\Delta$ connecting two vertices of the same color is called a \emph{ladderpole edge} (and is always of the opposite color). The other edges are called \emph{rungs}. It turns out that each vertex belongs to exactly two ladderpole edges, and ladderpole edges arrange into infinitely many, disjointly embedded simplicial lines of alternating colors, called \emph{ladderpoles}. Every rung connects two vertices from two consecutive ladderpoles. The region between two consecutive ladderpoles is called a \emph{ladder}. In each ladder, all triangles have their $\pi$-angle, or \emph{tip}, on the same side of their base rung (say \emph{above} the base rung if we arrange the ladder vertically with suitable orientation); in the next ladder the tips are on the other side (\emph{below} the base rungs). Ladders of the former type are called ascending, of the latter type descending; see Figure~\ref{link}.

Note that vertices of $\Delta$ have well-defined coordinates in the plane $\mathbb{C}$ (up to similarity), given by any developing map of the hyperbolic metric on $M_\varphi$ that takes $\Omega$ to $\infty\in\mathbb{P}^1\mathbb{C}=\partial_\infty \mathbb{H}^3$. However, we will usually draw the ladderpoles as vertical lines (with regular meanders to respect the train-track convention for angles $0$ and $\pi$ as in Figure \ref{link}), emphasizing the combinatorics at the expense of the geometry.

\subsection{The Cannon-Thurston map, and the Main Result} \label{sec:mainresult}
Let $\Sigma$ (a disk) be the universal cover of the fiber $S$ of the hyperbolic surface bundle $M=M_\varphi$, and~$\mathbb{S}$ (a circle) the natural boundary of $\Sigma$. 
The inclusion $S\rightarrow M$ lifts to a map $\iota: \Sigma \rightarrow \mathbb{H}^3$ between the universal covers. Bowditch \cite{bowditch}, generalizing work of Cannon and Thurston \cite{CT}, proved the following surprising fact.
\begin{factB} \label{fact:CT}
The\note{Caution: ``Fact~B'' later referred to manually} map $\iota$ extends continuously to a boundary map $\overline\iota:\mathbb{S}\rightarrow \mathbb{P}^1\mathbb{C}$ which is a \emph{surjection} from the circle to the sphere. The endpoints of any leaf of $\lambda^\pm$ (lifted to~$\Sigma$) have the same image under $\overline\iota$, and this in fact generates all the identifications occurring under $\overline\iota$. 
\end{factB}

In this note, we prove a correspondence between the combinatorics of the Cannon-Thurston map $\overline{\iota}$ (the ``order in which $\overline\iota$ fills out the sphere'') and the triangulation $\Delta$ of the plane. To state the correpondence, we will first need to prove facts about~$\overline{\iota}$ (details in Section~\ref{sec:CT}).

Recall the chosen parabolic fixed point $\infty$ of the Kleinian group $\Gamma$.
The surjection $\overline\iota:\mathbb{S} \rightarrow \mathbb{C} \cup \{\infty\}$ goes infinitely many times through the point $\infty$, by Fact~B (indeed there are infinitely many leaves terminating at a given parabolic boundary point of~$\Sigma$). We may imagine that $\overline\iota$ \emph{changes color} (red/blue) each time it goes through~$\infty$: the resulting coloring of the plane $\mathbb{C}$ becomes an interesting object to look at.

\begin{theorem} \label{thm:jordan}
There exists a $\mathbb{Z}$-family of Jordan curves $J_i$ of $\mathbb{C}\cup\{\infty\}$, bounding domains $D_i$, with the following properties:
\begin{itemize}
 \item For all $i$ the curve $J_i$ goes through $\infty$;
 \item $\dots \supset D_{-1} \supset D_0 \supset D_1 \supset D_2\supset\dots$;
 \item $\bigcap_{i\in\mathbb{Z}} D_i = \emptyset$ and $\bigcup_{i\in\mathbb{Z}} D_i = \mathbb{C}$; 
 \item $J_i \cap J_{i'} = \{\infty\}$ if and only if $|i'-i|> 1$;
 \item For every $i\in\mathbb{Z}$, the closure of $D_i\smallsetminus D_{i+1}$ is the union of a family $\{\delta^i_s\}_{s\in\mathbb{Z}}$ of closed disks, all disjoint except that each $\delta^i_s$ shares one boundary point with $\delta^i_{s+1}$;
 \item Between the $i$-th and $(i+1)$-st color switches, the map $\overline{\iota}$ fills out the $\delta^i_s$ one by one; the order of filling switches with the parity of $i$.
\end{itemize}
\end{theorem}

\begin{figure}[h!]
\centering
\labellist
\small\hair 2pt
\pinlabel {$\color{medgreen} J_0$} [c] at 10 5
\pinlabel {$\color{purple} J_1$} [c] at 27 5
\pinlabel {$\color{medgreen} J_2$} [c] at 44 5
\pinlabel {$\color{purple} J_3$} [c] at 61 5
\pinlabel {$\color{medgreen} J_4$} [c] at 78 5
\pinlabel {$\color{purple} J_5$} [c] at 95 5
\pinlabel $\delta$ [c] at 52 58
\pinlabel $\bullet$ [c] at 52.5 80
\pinlabel $\bullet$ [c] at 52.5 28
\pinlabel $\circ$ [c] at 39 36.75
\pinlabel $\circ$ [c] at 39 59.7
\pinlabel $\circ$ [c] at 39 71.25
\pinlabel $\circ$ [c] at 66 71
\pinlabel $\varepsilon$ [c] at 43 48
\pinlabel $\varepsilon'$ [c] at 46.5 34
\pinlabel {$\varepsilon$  in-furrow edge} [l] at 105 55
\pinlabel {$\varepsilon'$  cross-furrow edge} [l] at 104 47
\pinlabel {$\bullet$ gate of $\delta$} [l] at 104 39
\pinlabel {$\circ$ spike of $\delta$} [l] at 104 31
\endlabellist
\includegraphics[width=7cm]{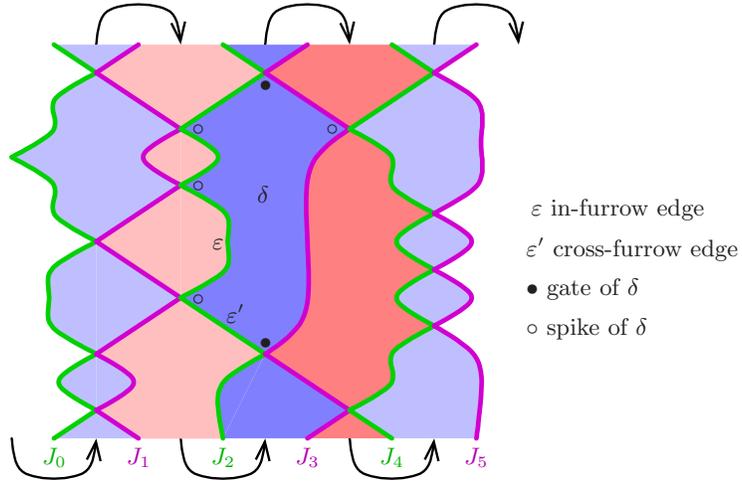}
\caption{Six Jordan curves $J_i$ through $\infty$ bound 5 consecutive furrows, alternatingly red and blue. 
The domain $D_i$ is the area to the right of $J_i$. 
The disk $\delta$ belongs to the furrow $D_2\smallsetminus D_3$.
Arrows materialize in which order furrows are filled out by $\overline{\iota}$. 
}
\label{furrow}
\end{figure} 

By this theorem, the trajectory of the plane-filling curve $\overline{\iota}$ is reminiscent of that of a plowing ox (or boustrophedon, the name of an ancient writing style): we consequently call the closure of $D_i\smallsetminus D_{i+1}$ a \emph{furrow}; see Figure~\ref{furrow}. The disks $\{\delta^i_s\}_{(i,s)\in\mathbb{Z}^2}$ making up all the furrows define a tessellation of $\mathbb{C}$ in which every vertex has order $4$, and is adjacent to two consecutive disks of the $i$-th furrow, one disk of the $(i-1)$-st furrow and one disk of the $(i+1)$-st furrow (for some $i$). Each disk $\delta$ of the $i$-th furrow has:
\begin{itemize}
 \item 2 vertices (the \emph{gates} of $\delta$) shared with other disks of the $i$-th furrow;
 \item some nonnegative number of vertices (called \emph{spikes} of $\delta$) shared with disks of the $(i+2)$-nd or $(i-2)$-nd furrow.
\end{itemize}
Edges of the tessellation always separate disks from consecutive furrows; they come in two types (Figure \ref{furrow}):
\begin{itemize}
 \item \emph{in-furrow} edges, connecting consecutive gates of the same furrow (or equivalently, two consecutive spikes of some disk of the adjacent furrow);
 \item \emph{cross-furrow} edges, connecting two gates of adjacent furrows (or equivalently, the first or last spike on one side of some disk to the adjacent gate).
\end{itemize}

With this terminology, we can state our main result, which consists of (1) a full dictionary between the various features of the two tessellations, and (2) a recipe book to reconstruct (combinatorially) one tessellation from the other.
\begin{theorem} \label{thm:dico}
{\bf (1) Dictionary.}
The Cannon-Thurston tessellation and the link $\Delta$ of the Agol triangulation have the same vertex set, and there are natural bijections between the following objects:

\begin{center}
\begin{tabular}{r|l}
\emph{Cannon-Thurston tessellation} & \emph{Triangulation $\Delta$} \\ \hline
 Vertices & Vertices \\
 Furrows & Ladderpoles \\
 Disks & Ladderpole edges \\
 Spikes & Rungs \\
$ \left . \begin{array}{l} \text{In-furrow} \\ \text{Cross-furrow} \end{array} \right \}$ edges &
$ \left . \begin{array}{l} \text{Non-hinge} \\ \text{Hinge} \end{array} \right \}$ triangles
\end{tabular}\end{center}

{\bf (2) Recipe book.} 
Given the link $\Delta$, we can obtain topologically the 1-skeleton of the Cannon-Thurston tessellation by drawing, for each triangle, an arc from its tip to the tip of the next triangle across the base rung (this may create double edges along the ladderpoles: keep them). See Figure~\ref{translate}, left.

Conversely, given the Cannon-Thurston tessellation, we can obtain the 1-skeleton of the triangulation $\Delta$ by adding edges connecting each gate of a blue (resp.\ red) cell to all the vertices clockwise (resp.\ counterclockwise) until the other gate, and deleting redundant edges. See Figure~\ref{translate}, right.
\end{theorem}

\begin{figure}[h!]
\centering
\labellist
\small\hair 2pt
\endlabellist
\includegraphics[width=12cm]{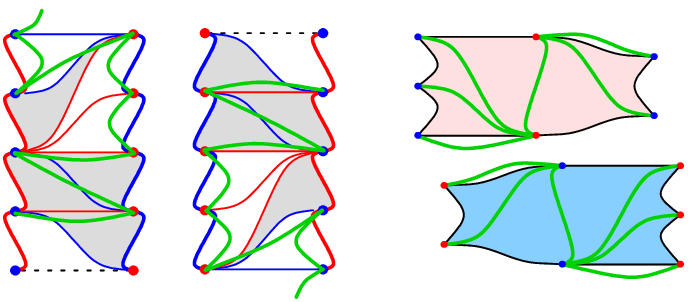}
\caption{\emph{Left}: two ladders (ascending and descending) of the Agol triangulation, with, superimposed in green, the tip-to-tip edges to be drawn to obtain the topological 1-skeleton of the Cannon-Thurston tessellation. Hinge triangles are shaded.
\newline
\emph{Right}: two disks (from a red and a blue furrow) of the Cannon-Thurston tessellation,  with, superimposed in green, the edges to be inserted to obtain the Agol cusp triangulation (before deletion of redundant edges).}
\label{translate}
\end{figure} 

In the theorem above, the redundant edges to be deleted in the last step can be either present in the initial Cannon-Thurston tessellation (namely, the boundary of a cell $\delta^i_s$ with no spikes consists of two mutually isotopic edges along a ladderpole; Figure \ref{furrow} shows four such ``small'' cells) --- or they can be created during the process of adding extra edges.

\begin{remark}
As mentioned above, the vertices of the two tessellations of Theorem \ref{thm:dico} are well-defined complex algebraic numbers. Due to the orientation issue raised in \cite{Issa}, however, edges of $\Delta$ must be thought of combinatorially, not as straight segments. On the other hand, edges of the Cannon-Thurston tessellation are precise, fractal-looking plane curves: some examples are beautifully rendered in \cite{CD, Dicks-Sakuma}.
\end{remark}

\subsection*{Notation}
Throughout the paper, we will denote by $\Sigma$ the universal cover of the fiber $S$ of the surface bundle $M_\varphi$, and by $\overline{\Sigma}$ and $\overline{S}$ the metric completions of $\Sigma$ and $S$ respectively, for the locally Euclidean metric. There is a commutative diagram
\begin{equation} \label{com}
 \begin{array}{ccc} \Sigma & \hookrightarrow & \overline{\Sigma} \\ \downarrow && \downarrow \\ S & \hookrightarrow & \overline{S}. \end{array}
 \end{equation}
However, note that the righmost vertical map $\overline{\Sigma} \rightarrow \overline{S}$ is \emph{not} a universal covering: it has infinite branching above all the points representing punctures of $S$.

\subsection{Plan of the paper}
In Section~\ref{sec:triangulation} we prove Theorem~\ref{thm:veering}. In Section~\ref{sec:normalforms} we study geodesics in a half-translation surface to produce a combinatorial description of the source circle $\mathbb{S}$ of the Cannon-Thurston map $\overline{\iota}$. In Section~\ref{sec:CT} we use this understanding to prove Theorem~\ref{thm:jordan} on the combinatorics of $\overline{\iota}$. Finally, in Section~\ref{sec:dico} we prove Theorem~\ref{thm:dico}, making each line of the ``dictionary'' correspond to a certain type of rectangle in the foliated surface $(\Sigma,\lambda^\pm)$.


\section{The canonical veering triangulation} \label{sec:triangulation}

We now prove Theorem~\ref{thm:veering}. Let $S,\lambda^\pm, \varphi$ be as in the theorem, and $\overline{S}, \Sigma, \overline{\Sigma}$ as in \eqref{com}.
Let $g_0$ be a singular Euclidean metric on $\overline{\Sigma}$ that makes the measured foliations $\lambda^+$ and $\lambda^-$ vertical and horizontal, and $g_t:=(^{e^t}{}_{e^{-t}})g_0$.

By a singularity-free rectangle in $\overline{\Sigma}_t:=(\overline{\Sigma}, g_t)$, we mean an embedded rectangle whose sides are leaf segments of $\lambda^{\pm}$ and which contains no singularity except possibly in its boundary. 
Note that a singularity-free rectangle contains \emph{at most} one singularity in each edge: indeed no leaf of $\lambda^+$ or $\lambda^-$ can connect two singularities, otherwise there would be arbitrarily short such leaves (by applying $\varphi^{\pm 1}$ many times), contradicting the fact that the singular set is finite in the quotient $\overline{S}$.

A singularity-free rectangle in $\overline{\Sigma}_t$ receives a height and a width (depending on~$t$) from the transverse measures on $\lambda^\pm$. We speak of a singularity-free \emph{square} if the height and width are equal.

The following construction is an analogue of the Delaunay triangulation relative to the singular set, with circles replaced by squares.\footnote{An analogous construction works with circles (or squares) replaced by any convex, centrally symmetric plane shape $\mathcal{C}$, provided no straight segment between singularities is parallel to a segment of $\partial \mathcal{C}$. See \cite{Klyachin} for related ideas, as well as \cite{L1} and references therein.} In the context of this paper, we will just refer to it as the Delaunay cellulation.
\begin{proposition}
Connecting the singularities found in the boundary of every maximal singularity-free square of $\overline{\Sigma}_t$ produces, in the quotient $\overline{S}_t$, finitely many triangles and (exceptionally) quadrilaterals, which define a cellulation of $\overline{S}_t$.
\end{proposition}
\begin{proof}
First, there exist maximal singularity-free squares in $\overline{\Sigma}_t$: to find one, start with a singularity $p$; construct a small square $Q$ containing $p$ in the interior of one edge; then scale $Q$ up with respect to $p$ until $Q$ bumps into another singularity~$p'$. If $p'$ belongs to the interior of the edge opposite to $p$ then $Q$ is maximal. Otherwise, $Q$ has a corner $q$ whose two adjacent closed edges contain $\{p,p'\}$ in their union; scale $Q$ further up with respect to $q$ until $Q$ bumps into a third singularity $p''$, necessarily in one of the two remaining edges. Then $Q$ is maximal. (For certain values of $t$, there may simultaneously appear a fourth vertex $p'''$ in the last edge.)

We call the convex hull of the singularities contained in the boundary of a maximal singularity-free square a \emph{Delaunay cell}. Delaunay cells can be either edges, or triangles, or quadrilaterals (the latter do not occur for generic $t$). We claim that:
\begin{equation}
\begin{array}{l} \text{\textit{Two distinct Delaunay cells in $\overline{\Sigma}_t$ can only intersect}} \\ \text{\textit{ (if at all) along an edge or a vertex.}}
\end{array} \label{eq:inter}
\end{equation}

To see this, consider two maximal singularity-free squares $Q=ABCD$ and $Q'=A'B'C'D'$ in $\overline{\Sigma}_t$. If $Q$ and $Q'$ have disjoint interiors, the conclusion is immediate. Otherwise, denote by $D_Q, D_{Q'}$ the Delaunay cells in $Q$ and $Q'$ respectively.
Up to permuting $Q, Q'$, rotating and relabelling the vertices, we are in one of the following 5 situations (see Figure~\ref{delaunay}):
\begin{figure}[h!]
\centering
\labellist
\small\hair 2pt
\pinlabel $A$ [c] at 5 40
\pinlabel $B$ [c] at 39 41
\pinlabel $C$ [c] at 39 8
\pinlabel $D$ [c] at 4 6
\pinlabel $A'$ [c] at 16 27
\pinlabel $B'$ [c] at 49 27
\pinlabel $C'$ [c] at 49 2
\pinlabel $D'$ [c] at 14 2
\pinlabel $E$ [c] at 39 30
\pinlabel $F$ [c] at 14 13
\pinlabel $A$ [c] at 61 41
\pinlabel $B$ [c] at 95 42
\pinlabel $C$ [c] at 95 8
\pinlabel $D$ [c] at 60 6
\pinlabel $A'$ [c] at 80 32
\pinlabel $B'$ [c] at 105 33
\pinlabel $C'$ [c] at 105 16
\pinlabel $D'$ [c] at 80 15
\pinlabel $E$ [c] at 89 32
\pinlabel $F$ [c] at 89 15
\pinlabel $A$ [c] at 117 41
\pinlabel $B$ [c] at 152 42
\pinlabel $C$ [c] at 148 6
\pinlabel $D$ [c] at 115 6
\pinlabel $A'$ [c] at 132 32
\pinlabel $B'$ [c] at 162 33
\pinlabel $C'$ [c] at 159 6
\pinlabel $D'$ [c] at 133 6
\pinlabel $E$ [c] at 145 29
\pinlabel $F$ [c] at 140 5
\pinlabel $A$ [c] at 174 41
\pinlabel $B$ [c] at 209 42
\pinlabel $C$ [c] at 205 6
\pinlabel $D$ [c] at 171 6
\pinlabel $A'$ [c] at 190 32
\pinlabel $B'$ [c] at 221 33
\pinlabel $C'$ [c] at 217 6
\pinlabel $D'$ [c] at 193 6
\pinlabel $E$ [c] at 202 28
\pinlabel $A$ [c] at 231 40
\pinlabel $B$ [c] at 259 40
\pinlabel $C$ [c] at 263 6
\pinlabel $D$ [c] at 227 6
\pinlabel $A'$ [c] at 238 40
\pinlabel $B'$ [c] at 272 40
\pinlabel $C'$ [c] at 275 6
\pinlabel $D'$ [c] at 244 6
\pinlabel $E$ [c] at 251 40
\pinlabel $F$ [c] at 255 5
\pinlabel $(1)$ [c] at 20 49
\pinlabel $(2)$ [c] at 77 49
\pinlabel $(3)$ [c] at 134 49
\pinlabel $(4)$ [c] at 191 49
\pinlabel $(5)$ [c] at 255 49
\endlabellist
\includegraphics[width=13cm]{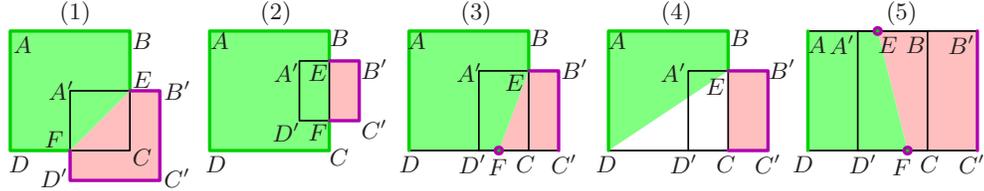}
\caption{Five possible relative positions of maximal singularity-free squares $Q=ABCD$ and $Q'=A'B'C'D'$. The subsets of $\partial Q$ and $\partial Q'$ that may contain singularities are marked by thicker lines, respectively green and purple. Convex regions that contain the Delaunay polygons $D_Q$ and $D_{Q'}$ are shaded.}
\label{delaunay}
\end{figure} 
\begin{enumerate}
 \item 
 The open segment $(BC)$ intersects $(A'B')$ at a point $E$, and $(CD)$  intersects $(A'D')$ at a point $F$. Then all singularities in $\partial Q$ belong to the broken line $[EBADF]$, so $D_Q$ is contained in the pentagon $EBADF$; similarly $D_{Q'}\subset EB'C'D'F$. These pentagons share just one edge $EF$, hence the result.
 \item 
 The open segment $(BC)$ intersects $(A'B')$ at a point $E$ and $(C'D')$ at a point $F$, with $B,E,F,C$ lined up in that order. Then all singularities in $\partial Q$ belong to the broken line $[EBADCF]$, so $D_Q$ is contained in its convex hull $ABCD$; similarly $D_{Q'}\subset EB'C'F$. These rectangles share just one edge $EF$, hence the result.
 \item 
 The open segment $(BC)$ intersects $(A'B')$ at a point $E$, the points $D$, $D'$, $C$, $C'$ are lined up in that order, and $[DC]$ contains a singularity $F$. Since the leaves of $\lambda^\pm$ through $F$ contain no other singularity than $F$, the singularities of $\partial Q$ (resp.\ $\partial Q'$) lie in the broken line $[FDABE]$ (resp.\ $[FC'B'E]$). The convex hulls of these broken lines are polygons sharing just one edge $EF$.
 \item
 The open segment $(BC)$ intersects $(A'B')$ at a point $E$, the points $D$, $D'$, $C$, $C'$ are lined up in that order, and $[DC]$ contains \emph{no} singularity. The singularities of $\partial Q$ (resp.\ $\partial Q'$) lie in the broken line $[DABE]$ (resp.\ $CC'B'E$). The convex hulls of these broken lines are polygons sharing just one vertex $E$.
\item
The points $A,A', B, B'$ are lined up in that order, and so are $D, D', C, C'$. Let $E$ be a point on $[AB']$, equal to the singularity if there is one on this segment. Let $F$ be a point on $[DC']$, equal to the singularity if there is one on this segment. The singularities in $\partial Q$ (resp.\ $\partial Q'$) are in the broken line $[EADF]$ (resp.\ $[EB'C'F]$). The convex hulls of these broken lines are polygons sharing just one edge $EF$. This proves \eqref{eq:inter}.
\end{enumerate}

Since the Delaunay polygons have disjoint interiors in $\overline{\Sigma}_t$, in particular \eqref{eq:inter} implies that the projections of these interiors in the quotient surface~$S_t$ are \emph{embedded} (not just immersed).

Next, we claim that every side $[AB]$ of a Delaunay polygon is a side of exactly one other Delaunay polygon, adjacent to the first one. This is clear if one considers the 1-parameter family of all squares in the Euclidean plane containing a pair of points $\{A, B\}$ in their boundary (Figure~\ref{bump}): $[AB]$ subdivides each square of the family into two regions, which vary monotonically (for the inclusion) in opposite directions with the parameter $\tau$. Thus for each side of $[AB]$ there is an extremal value of $\tau$ at which the region on that side bumps for the first time into one or (exceptionally) two singularities. The convex hull of the union of these singularities and $\{A, B\}$ is the Delaunay polygon on that side of $[AB]$.

\begin{figure}[h!]
\centering
\labellist
\small\hair 2pt
\pinlabel $A$ [c] at 26 61
\pinlabel $B$ [c] at 63 14
\endlabellist
\includegraphics[width=4.5cm]{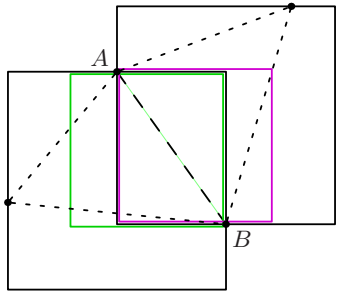}
\caption{The extremal squares circumscribed to a segment $[AB]$, and two intermediate squares (colored). Dotted: the Delaunay triangles containing $[AB]$.}
\label{bump}
\end{figure} 

The Delaunay polygons therefore define a cell decomposition of some region of~$\overline{S}$. This region is open and closed, because there are only finitely many Delaunay polygons in $\overline{S}$ (the diameter of $\overline{S}$ gives an upper bound on the possible sizes of singularity-free squares, so a compactness argument applies). Therefore, the Delaunay polygons give a cell decomposition (generically a triangulation, but possibly nonsimplicial) of $\overline{S}$ itself.
\end{proof}

\begin{remark} The ideal Delaunay decomposition of $S_t=(S, g_t)$ (obtained by removing the singularities of $\overline{S}$) varies with $t$. The changes occur at the values of~$t$ such that the decomposition contains a quadrilateral inscribed in a square $Q$ (or several such quadrilaterals). At such times $t$, the triangulation undergoes a \emph{diagonal exchange}: before $t$ the triangulation contains an edge connecting the singularities on the vertical sides of $Q$; after $t$ the triangulation contains an edge connecting the singularities on the horizontal sides of~$Q$.
\end{remark}

Interpreting such diagonal exchanges as (flattened) tetrahedra as in Figure~\ref{tetrahedron}, we obtain a so-called \emph{layered} ideal triangulation of $S\times \mathbb{R}$ which naturally descends to the quotient $S$-bundle $M_\varphi$, and is veering by construction if we just color edges in red or blue according to the sign of their slope in the translation surface $S$. Theorem~\ref{thm:veering} is proved.

An important result of \cite{Agol} is that there is at most one layered veering triangulation of $M_\varphi$, so it has to be the one constructed in Theorem \ref{thm:veering}.

\section{Normal forms of points of $\partial_\infty \Sigma$} \label{sec:normalforms}

In this section, $S$ is a half-translation surface with at least one puncture, no singularities (except at punctures), and capable of carrying a complete hyperbolic metric $h$, for which the punctures become cusps. The universal cover $\mathbb{H}^2$ of $(S,h)$ has a boundary at infinity which is a topological circle $\mathbb{S}$. This circle $\mathbb{S}$ is topologically independent of the choice of $h$, in the sense that if $h'$ is another hyperbolic metric on~$S$, then the identity map from $(S,h)$ to $(S,h')$ lifts to a self-homeomorphism of $\mathbb{H}^2$ which extends to a unique self-homeomorphism of $\mathbb{S}$. 

Alternatively we can obtain the circle $\mathbb{S}$ in the following way. The singular Euclidean surface $\overline{\Sigma}$ is Gromov-hyperbolic (quasi-isometric to the free group $\pi_1(S)$) and its boundary at infinity $\partial_\infty \overline{\Sigma}$ is a Cantor set. Moreover, $\partial_\infty \overline{\Sigma}$ carries a natural cyclic order induced by the orientation of the disk $\Sigma$, or its quotient $S$. The circle $\mathbb{S}$ is naturally identified with the quotient of the Cantor set $\partial_\infty \overline{\Sigma}$ under the equivalence condition $\sim$ that \emph{collapses any two points that are not separated by a third (distinct) point} for the cyclic order. Indeed, since any two such points are the attracting and repelling fixed points $\xi_c^+, \xi_c^- \in\partial_\infty(\pi_1 (S))$ of a peripheral element $c$ of $\pi_1(S)$, a holonomy representation $\rho$ of the hyperbolic metric $h$ on $S$ takes $c$ to a parabolic isometry $\rho(c)$ of $\mathbb{H}^2$ fixing a unique ideal point $\eta_c \in\partial_\infty\mathbb{H}^2$. The map $\{\xi_c^+, \xi_c^-\}\mapsto \eta_c$ extends naturally to a $\rho$-equivariant homeomorphism 
$$\psi:\partial_\infty(\pi_1(S))/\!\sim \: \tilde{\longrightarrow} \: \partial_\infty\mathbb{H}^2=\mathbb{S}.$$

In the proposition below, we describe the points of the circle $\mathbb{S}$ as ``normalized'' paths (possibly semi-infinite) in the singular Euclidean surface $\overline{\Sigma}$. The singular points of $\overline{\Sigma}$ are infinite branching points of the covering map towards the completion $\overline{S}$ of $S$: let $\Omega\in\overline{\Sigma}$ denote such a singular point, fixed throughout the paper.

\begin{proposition} \label{paths}
There exists a natural parameterization $\Psi$ of the circle $\mathbb{S}$ by the collection of all piecewise straight paths $\gamma$ in $\overline{\Sigma}$ starting at $\Omega$ and turning only at singularities (leaving angles $\geq \pi$ as they turn), possibly terminating at a singularity.

Given two distinct, nontrivial such paths $\gamma$ and $\gamma'$, the 3 ideal points $\Psi(\gamma)$, $\Psi(\gamma')$, $\Psi(\Omega)$ in $\partial_\infty\mathbb{H}^2$ (where the same symbol $\Omega$ is used to denote the ``trivial'' path that stays at $\Omega$) span an ideal triangle of $\mathbb{H}^2$. This triangle is clockwise oriented if and only if the trajectory of $\gamma'$ in $\overline{\Sigma}$ departs from that of $\gamma$ to the right, possibly after a finite common prefix. This must happen at a singularity which the paths may pass on either side, or stop at: see Figure~\ref{clockwise}.
\end{proposition}
\begin{figure}[h!]
\centering
\labellist
\small\hair 2pt
\pinlabel $\gamma'$ [c] at 63 31
\pinlabel $\gamma$ [c] at 61 45
\pinlabel $\gamma'$ [c] at 80 27
\pinlabel $\gamma$ [c] at 83 37
\pinlabel $\gamma'$ [c] at 97 39
\pinlabel $\gamma$ [c] at 120 39
\pinlabel $\gamma'$ [c] at 134 30
\pinlabel $\gamma$ [c] at 147 26
\pinlabel $\gamma'$ [c] at 168 45
\pinlabel $\gamma$ [c] at 161 30
\endlabellist
\includegraphics[width=11cm]{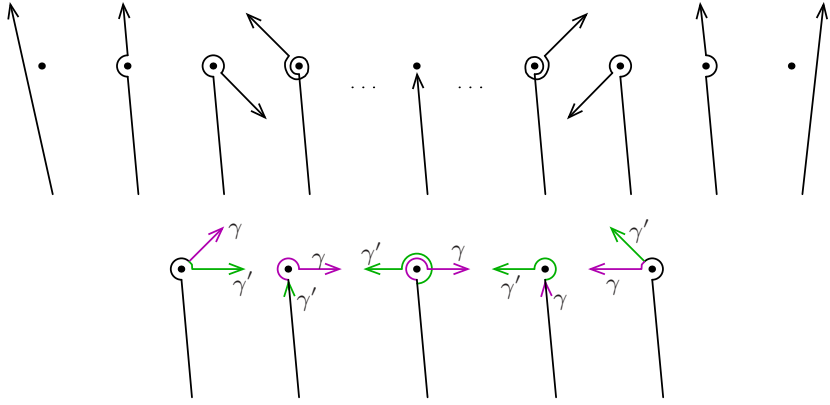}
\caption{Illustration of Proposition~\ref{paths}. \emph{Top}: a sequence of paths passing through a singularity and progressing clockwise in the space $\mathbb{S}$ of paths. \emph{Bottom}: at a singularity, a path $\gamma'$ (green) departs from another path $\gamma$ (purple) to the right after a nonempty common prefix (black), in all of 5 possible ways.}
\label{clockwise}
\end{figure} 
\begin{proof}[Proof of Proposition~\ref{paths}]
The map $\psi$ already identifies $\mathbb{S}$ with the quotient space $\partial_\infty(\pi_1(S))/\!\sim$. Let us identify the latter set with a class of paths starting at $\Omega$. 

We can find a connected, polygonal fundamental domain $P$ of $\overline{\Sigma}$ with all vertices at singularities, for example by taking an appropriate union of cells of one of the Delaunay cellulations of Section \ref{sec:triangulation}. Up to choosing a fixed connecting path between $\Omega$ and a lift of the basepoint of $S$, we can view $\partial_\infty(\pi_1(S))$ as the space of infinite sequences $(P_i)_{i\geq 0}$ of distinct copies of $P$ in $\overline{\Sigma}$ such that $\partial P_0$ contains $\Omega$ and each $P_i$ shares an edge with $P_{i+1}$.

If $\xi \in \partial_\infty(\pi_1(S))/\!\sim$ does not come from a peripheral fixed point, then we can see $\xi$ as a (unique) infinite path of polygons $(P_i)_{i\geq 0}$ as above, such that any singularity belongs to at most finitely many of the $P_i$. In particular the $P_i$ escape any compact set of $\overline{\Sigma}$. Since $\overline{\Sigma}$ is a CAT(0) space, there exists a unique, infinite geodesic path of the singular Euclidean surface $\overline{\Sigma}$, issued from $\Omega$, and following the $P_i$.

If $\xi \in \partial_\infty(\pi_1(S))/\!\sim$ does come from a peripheral fixed point, then we can still see $\xi$ as an infinite path of polygons $(P_i)_{i\geq 0}$, but there exists a (smallest) $i_0$ such that all the $P_i$ for $i\geq i_0$ share a certain vertex $\Omega'$. The other fixed point of the peripheral is obtained by replacing this suffix $(P_i)_{i\geq i_0}$, formed of a sequence of polygons turning around $\Omega'$ in one direction, by the sequence turning in the other direction. A pair $\{\xi^+, \xi^-\}$ of peripheral fixed points as above can thus naturally be identified with the geodesic path in $\overline{\Sigma}$ from $\Omega$ to $\Omega'$, following the polygons $P_0, \dots, P_{i_0}$ and \emph{terminating} at $\Omega'$. 

A local minimizing argument shows that the geodesics of $\overline{\Sigma}$ are exactly the piecewise straight curves $\gamma$ that turn only at singularities, with all turning angles being $\geq \pi$. By the above discussion, the space of finite (res.\ infinite) such geodesics $\gamma$ issued from $\Omega$ identifies with the unordered pairs (resp.\ singletons) in the domain $\partial_\infty(\pi_1(S))/\!\sim$ of the map $\psi$. Let us call $\Psi:\gamma \mapsto \text{(endpoint of $\gamma$)}\in\mathbb{S}$ the composition of this identification with $\psi$. By construction, $\Psi$ is bijective.

To check the statement on clockwise orientations, pick points $p\neq p'$ in $\mathbb{S}\smallsetminus \{\Psi(\Omega)\}$, and let $\mathcal{L}$ be any oriented curve in $\mathbb{H}^2$ from $\Psi(\Omega)$ to $p$. Then $p'$ lies to the right of $p$ as seen from $\Psi(\Omega)$ if and only if there exists a path from $p'$ to a point of $\mathcal{L}$ hitting $\mathcal{L}$ on the right side.

Next, let $\gamma$ and $\gamma'$ be the geodesics in $\overline{\Sigma}$ corresponding to $p$ and $p'$. These geodesics have a common prefix (possibly reduced to $\{\Omega\}$) before they diverge at some puncture. Let $\varepsilon>0$ be a small real number and $\mathcal{N}$ the $\varepsilon$-neighborhood of the singular set of $\overline{\Sigma}$. Deform $\gamma$ and $\gamma'$ to curves $\eta$ and $\eta'$ that coincide with $\gamma$ and $\gamma'$ outside the closure of $\mathcal{N}$, but follow circle arcs in $\partial \mathcal{N}$ at each turning singularity. We can view $\eta, \eta'$ as curves in the disk $\Sigma\simeq \mathbb{H}^2$ (escaping to parabolic fixed points if $\gamma, \gamma'$ terminate), and $\eta'$ leaves $\eta$ to the right if and only if $\gamma'$ leaves $\gamma$ to the right. The criterion above (with $\mathcal{L}=\eta$) gives the result.
\end{proof}
In the remainder of this paper, the word \emph{path} will usually refer to a geodesic path in $\overline{\Sigma}$, issued from the singular point $\Omega$. The topology on the space of paths is induced by its identification with $\mathbb{S}$. It may also be described thus: \emph{two paths are close if, when followed from $\Omega$, they stay close in $\overline{\Sigma}$ for a great amount of length}.

\section{The combinatorics of the Cannon-Thurston map $\overline{\iota}$} \label{sec:CT}

In this section we prove Theorem~\ref{thm:jordan} about the combinatorics of the Cannon-Thurston map $\overline{\iota}$. Our tools are Proposition~\ref{paths} for the description of the domain of~$\overline{\iota}$ (paths) with its cyclic order, and Fact~B for the fibers of $\overline{\iota}$.

As in the previous section, let $\Omega$ be a singularity of $\overline{\Sigma}$ (the metric completion of the universal cover of the flat punctured surface $S$), and let the circle $\mathbb{S}$ be the space of $\overline{\Sigma}$-geodesic paths $\gamma$ issued from $\Omega$. By Fact~B, the points $\gamma$ of $\mathbb{S}$ that have the same image as $\Omega$ under the map $\overline{\iota}$ (the ``$\overline{\iota}$-fiber of $\Omega$'') are $\Omega$ itself, and all the rays shooting out from $\Omega$ along a leaf of $\lambda^+$ or $\lambda^-$. Since $\Omega$ is a branching point of infinite order, these rays form a $\mathbb{Z}$-family, counting clockwise: for any integer $i$, the the $i$-th ray $\ell_i$ is vertical for $i$ even and horizontal for $i$ odd.

\subsection{Fibers of $\overline{\iota}$ and colors}
We apply a M\"obius map to normalize so that $\overline{\iota}(\Omega)=\infty \in \mathbb{P}^1\mathbb{C}$.

Every point $\gamma$ of $\mathbb{S}$ that is \emph{not} in the fiber of $\Omega$ 
falls inbetween $\ell_i$ and $\ell_{i+1}$ for a unique integer $i$, which is even (resp.\ odd) if and only if the initial segment of (the geodesic representative of) $\gamma$ has positive (resp.\ negative) slope. We say that the directions between $\ell_i$ and $\ell_{i+1}$ form a \emph{quadrant} at $\Omega$. There are countably many quadrants.

Therefore, the rule that $\overline{\iota}$ changes color at each passage through $\infty$ means that the color of $\overline{\iota}(\gamma)$ is determined by the sign of the slope of the initial segment of~$\gamma$, or equivalently, by the parity of the quadrant that contains this initial segment.

To understand the interfaces between the two colors in the image of $\overline{\iota}$, we must therefore understand which $\overline{\iota}$-fibers (described by Fact~B) contain two paths $(\gamma, \gamma')$ whose initial segments belong to different quadrants at $\Omega$. In general, let $Q(\gamma)$ denote the quadrant containing the initial segment of $\gamma$. 

\begin{definition} \label{def:hook}
For any $i\in\mathbb{Z}$ and $t\in \overline{\mathbb{R}^+}:=[0,+\infty]$, let $\ell_i$ be the $i$-th leaf of $\lambda^\pm$ issued from $\Omega$ for the clockwise order.

The \emph{right} (resp.\ \emph{left}) \emph{$t$-hook along $\ell_i$}, written $\Gamma_i^+(t)$ (resp.\ $\Gamma_i^-(t)$), is the geodesic straightening of the path obtained by following the leaf $\ell_i$ from $\Omega$, for a length $t$, then making a right (resp.\ left) turn to follow a leaf of the other foliation --- all the way to infinity, or to another singularity. 
 
In particular, 
\begin{equation} \label{collapse}
 \Gamma_i^{\pm}(0)=\ell_{i\pm 1}~\text{ and }~\Gamma_i^\pm(+\infty)=\ell_i
 ~\text{ and }~ \overline{\iota}(\Gamma_i^+(t))=\overline{\iota}(\Gamma_i^-(t))
\end{equation}
where the last identity follows from Fact~B. Also, $Q(\Gamma_i^+(t))$ and $Q(\Gamma_i^-(t))$ are consecutive quadrants for any $0<t<+\infty$. 
\end{definition}

\begin{definition}
A singularity $p$ of $\Sigma$ is called a \emph{ruling} singularity if the geodesic from $\Omega$ to $p$ consists of a single straight Euclidean segment $[\Omega, p]$, equal to the diagonal of a singularity-free rectangle. We also call $[\Omega, p]$ a \emph{ruling segment}.
\end{definition}

\begin{proposition} \label{prop:shootoff}
Let $\mathcal{F}$ be a fiber of $\overline{\iota}$, not containing $\Omega$. One of the following holds:
\begin{itemize}
\item $Q(\mathcal{F})$ is a single quadrant;
\item $|Q(\mathcal{F})|=2$ and there exists a unique $(i,t)\in\mathbb{Z}\times \mathbb{R}_+^*$ such that $\Gamma_i^\pm(t) \in \mathcal{F}$;
\item $|Q(\mathcal{F})|=3$ and there exists a unique $(i,t,t')\in\mathbb{Z}\times \mathbb{R}_+^* \times \mathbb{R}_+^*$ such that $\{\Gamma_i^\pm(t), \Gamma_{i+1}^\pm(t')\} \subset \mathcal{F}$; this happens exactly when $\mathcal{F}$ contains a ruling segment $[\Omega,p]$, inscribed in a singularity-free rectangle of sidelengths $t, t'$.
\end{itemize}
\end{proposition}
\begin{proof}
By Fact~B, a fiber $\mathcal{F}$ may have cardinality $1$, $2$, or $\infty$. If $|\mathcal{F}|=1$ we are in the first case. 

\begin{figure}[h!]
\centering
\labellist
\small\hair 2pt
\pinlabel $\gamma'$ [c] at 55 44
\pinlabel $\gamma$ [c] at 110 47
\pinlabel $\Omega$ [c] at 81 15
\pinlabel $p$ [c] at 89 11
\pinlabel $x$ [c] at 71 54
\endlabellist
\includegraphics[width=11cm]{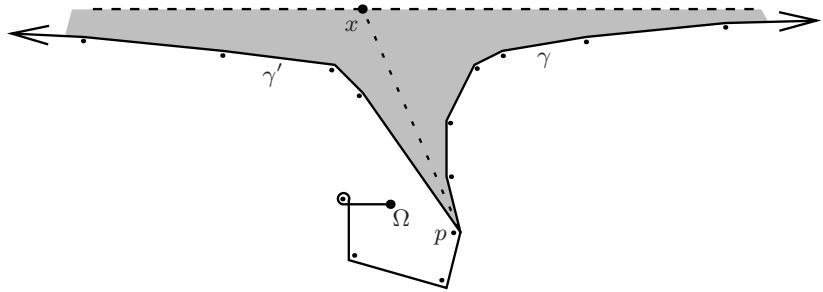}
\caption{An ideal triangle in $\overline{\Sigma}$, with one vertex $\Omega$ and the opposite edge equal to a horizontal leaf. We also draw the geodesic from $\Omega$ to $x$ (a nonsingular point on that leaf). In Figures~\ref{merge}--\ref{shootoff}--\ref{merge-2}--\ref{merge-3}, the shaded area is a portion of $\overline{\Sigma}$ containing no singularities.}
\label{merge}
\end{figure} 
Suppose now that $\mathcal{F}$ contains exactly two elements $\gamma, \gamma'$. By Fact~B, the paths $\gamma$ and $\gamma'$ are then geodesic representatives of two paths 
that coincide up to a point $x$ at which they shoot off on opposite rays of a leaf of $\lambda^+$ or $\lambda^-$ containing no singularity. We may assume this leaf is horizontal. The ideal triangle spanned by the endpoints of $\gamma, \gamma'$ and by their common origin $\Omega$ has one fully horizontal edge, and two edges whose directions (followed towards $\Omega$) depart monotonically from horizontal until they merge at a singular point $p$, and then continue on together (with arbitrary changes of direction) until the point $\Omega$. See Figure~\ref{merge}. 
\begin{figure}[h!]
\centering
\labellist
\small\hair 2pt
\pinlabel $\gamma'$ [c] at 58 28
\pinlabel $\gamma$ [c] at 108 31
\pinlabel $\ell_i$ [c] at 82 50
\pinlabel $p=\Omega$ [c] at 76 2
\pinlabel $t$ [c] at 82 25
\pinlabel $x$ [c] at 87 39
\endlabellist
\includegraphics[width=10.5cm]{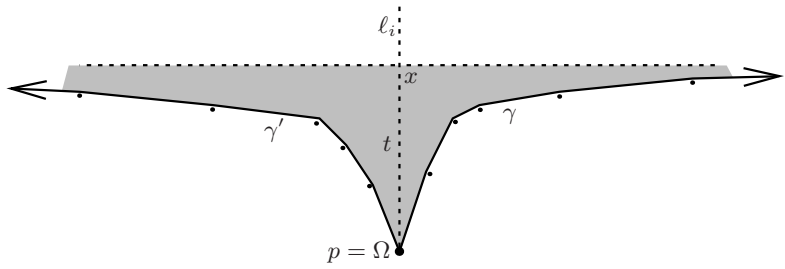}
\caption{Illustration of Proposition~\ref{prop:shootoff} when $|\mathcal{F}|=2$.}
\label{shootoff}
\end{figure} 
The total change of direction of~$\gamma$ between~$p$ and its endpoint at infinity of~$\overline{\Sigma}$, plus the total change of direction of~$\gamma'$ between~$p$ and its endpoint at infinity of~$\overline{\Sigma}$, is less than~$\pi$.  If the initial quadrants $Q(\gamma)$ and $Q(\gamma')$ at the point $\Omega$ are distinct, then they are consecutive and the only possibility is that $p$ coincides with $\Omega$ and the segment $[p,x]$ can be taken vertical along a leaf $\ell_i$. We then have the situation of Figure~\ref{shootoff}: the number $t$ is just the vertical distance $px$, and $\mathcal{F}=\{\Gamma_i^+(t), \Gamma_i^-(t)\}=\{\gamma, \gamma'\}$; the second case of Proposition~\ref{prop:shootoff} holds.

It remains to treat the case $|\mathcal{F}|=\infty$. By Fact~B, $\mathcal{F}$ then contains exactly: 
\begin{itemize}
 \item one path $\gamma$ terminating at a singularity $p$, and
 \item all the paths $\hat{\gamma}$ obtained from $\gamma$ by tacking on a leaf $l_s$ of $\lambda^+$ or $\lambda^-$ issued from $p$ (here $s$ ranges over $\mathbb{Z}$, the order being clockwise as seen from $p$).
\end{itemize}
\begin{figure}[h!]
\centering
\labellist
\small\hair 2pt
\pinlabel $\gamma'$ [c] at 36 41
\pinlabel $\gamma$ [c] at 111 45
\pinlabel $\gamma''$ [c] at 117 18
\pinlabel $\ddot{\gamma}$ [c] at 143 43
\pinlabel $\dot{\gamma}$ [c] at 112 65
\pinlabel $\Omega$ [c] at 76 6
\pinlabel $p$ [c] at 128 54
\pinlabel $t$ [c] at 78 34
\pinlabel $\ell_i$ [c] at 81 62
\pinlabel $x$ [c] at 78 48
\pinlabel $l_{s_0}$ [c] at 130 23
\pinlabel $l_{s_0+1}$ [c] at 95 55
\endlabellist
\includegraphics[width=10.5cm]{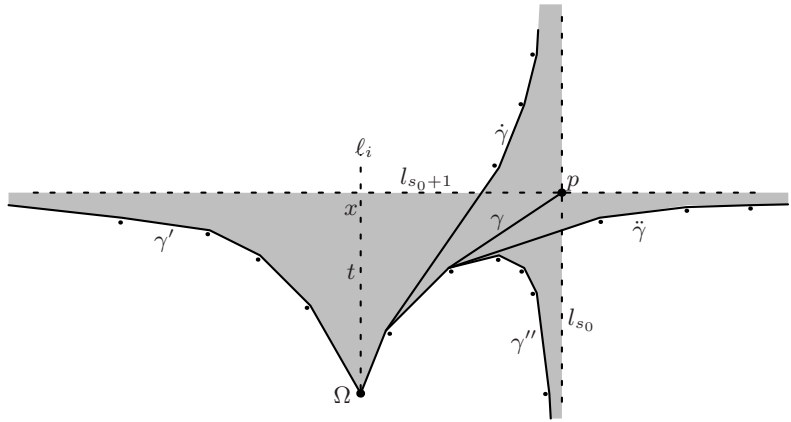}
\caption{Possibilities for the path $\hat{\gamma}$ include $\gamma', \gamma'', \dot{\gamma}, \ddot{\gamma}$.}
\label{merge-2}
\end{figure} 
If $|Q(\mathcal{F})|>1$, then in particular one of these paths, $\gamma'$, satisfies $Q(\gamma')\neq Q(\gamma)$.
It follows that $\{\gamma, \gamma'\}=\{\Gamma_i^+(t), \Gamma_i^-(t)\}$ for some $(i,t)\in \mathbb{Z}\times \mathbb{R}_+^*$: the argument is similar to the case $|\mathcal{F}|=2$ above, except that the horizontal leaf through $x$ terminates at~$p$ (Figure~\ref{merge-2}, ignoring for the moment 
the paths labelled $\dot{\gamma}, \ddot{\gamma}, \gamma''$).

It remains to find the quadrants $Q(\hat \gamma)$ for the other elements $\hat\gamma$ of the fiber~$\mathcal{F}$, obtained by tacking on to $\gamma$ a leaf $l_s$ of $\lambda^\pm$ issued from $p$.
Only four paths $\hat\gamma \in \mathcal{F}\smallsetminus \{\gamma\}$ have geodesic representatives that \emph{do not go} through $p$: these are the paths $\gamma'$, $\gamma''$, $\dot{\gamma}$, $\ddot{\gamma}$ shown in Figure~\ref{merge-2}. These possibilities correspond to tacking on to $\gamma$ any one of four consecutive leaves $l_{s_0-1}, l_{s_0}, l_{s_0+1}, l_{s_0+2}$:
\begin{itemize} 
\item either a boundary leaf (say $l_{s_0+1}$ or $l_{s_0}$) of the quadrant at $p$ containing the last segment of $\gamma$ (this yields geodesic representatives $\gamma'$ and $\gamma''$);
\item or one of the next closest leaves $l_{s_0+2}, l_{s_0-1}$ (this yields $\dot{\gamma}$ and $\ddot{\gamma}$).
\end{itemize}

All other possible $\hat \gamma$ go through $p$, and in particular satisfy $Q(\hat \gamma)=Q(\gamma)$. Closer inspection shows that actually $Q(\dot{\gamma})$ and $Q(\ddot{\gamma})$ are equal to $Q(\gamma)$, too (see Figure\ref{merge-2}).
The remaining possibility, $\hat \gamma=\gamma''$, can give two outcomes: 
\begin{itemize}
 \item If there is no ruling segment $[\Omega, p]$ (Figure~\ref{merge-2}), then $Q(\gamma'')=Q(\gamma)$ and we are still in the second case of Proposition~\ref{prop:shootoff}.
 \item If $[\Omega, p]$ \emph{is} a ruling segment, then $Q(\gamma'')\neq Q(\gamma)$. This situation is portrayed in Figure~\ref{merge-3}, and corresponds exactly to the third case of Proposition~\ref{prop:shootoff}.
\end{itemize}
\begin{figure}[h!]
\centering
\labellist
\small\hair 2pt
\pinlabel $\gamma'$ [c] at 36 72
\pinlabel {$\hat{\gamma} = \gamma''$} [c] at 97 25
\pinlabel $\gamma$ [c] at 110 71
\pinlabel $\Omega$ [c] at 76 37
\pinlabel $p$ [c] at 128 85
\pinlabel $\ell_i$ [c] at 81 92
\pinlabel $t$ [c] at 78 61
\pinlabel $\ell_{i+1}$ [c] at 137 38
\pinlabel $t'$ [c] at 105 41
\pinlabel $l_{s_0}$ [c] at 130 51
\pinlabel $l_{s_0+1}$ [c] at 95 86
\endlabellist
\includegraphics[width=11cm]{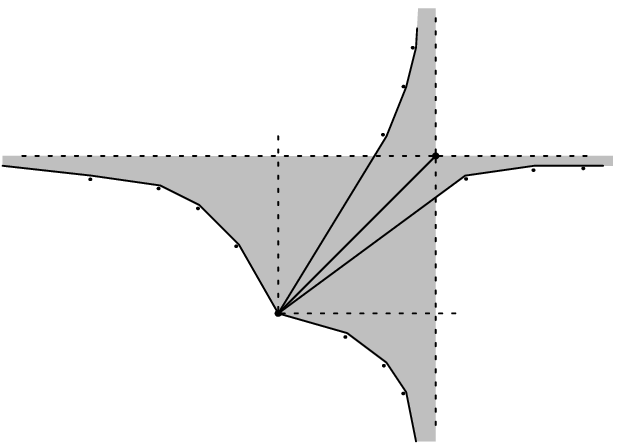}
\caption{When $[\Omega, p]$ is a ruling segment.}
\label{merge-3}
\end{figure} 
This concludes the proof of Proposition~\ref{prop:shootoff}.
\end{proof}

\subsection{Color interfaces of the Cannon-Thurston map}
We will use Proposition~\ref{prop:shootoff} to prove Theorem~\ref{thm:jordan} concerning the combinatorics of the Cannon-Thurston map $\overline{\iota}$. Here is a first step.

\begin{proposition} \label{prop:jordan1}
Recall $\overline{\mathbb{R}^+}:=[0,+\infty]$.
For all $i\in\mathbb{Z}$, the map 
$$\begin{array}{rrcl} 
& \overline{\mathbb{R}^+} & \longrightarrow & \mathbb{P}^1\mathbb{C} \\ 
J_i: & t & \longmapsto & \overline{\iota}(\Gamma^+_i(t)) 
\end{array}$$ 
is continuous, injective on $[0,+\infty)$, and $J_i(\overline{\mathbb{R}^+})$ is a Jordan curve.
\end{proposition}
(In this proposition, the image of $J_i$ is the curve labelled ``$J_i$'' in Figure~\ref{furrow}. 
The $t$-hook $\Gamma_i^+(t)$ in the definition of $J_i$ could be replaced by $\Gamma_i^-(t)$, due to \eqref{collapse} in Definition \ref{def:hook}.)
\begin{proof}
We first prove continuity. To fix ideas, supose the leaf $\ell_i$ shoots off from $\Omega$ vertically, upwards. We write $\Gamma(t)$ for $\Gamma_i^+(t)$.

If $t\in (0,+\infty)$ and $\Gamma(t)$ does not terminate on a singularity, then the geodesic representative of $\Gamma(t)$ makes infinitely many turns at singularities $p_1, p_2, \dots$. Continuity of $J_i$ at $t$ then follows from continuity of $\overline{\iota}$: indeed, for any integer $N$, if $t'$ is close enough to $t$ then the geodesic representative of $\Gamma(t')$ will coincide with that of $\Gamma(t)$ at least up to $p_N$. 

If $t\in(0,+\infty)$ and $\Gamma(t)$ terminates on a singularity $p$, consider $t'$ very close to~$t$. Suppose first that $t'<t$. Let $\Gamma(t^-)$ be the path obtained from $\Gamma(t)$ by tacking on a leaf of $\lambda^\mp$ making an angle $\pi$ with $\Gamma(t)$ \emph{below} $p$. Since $\overline{\iota}(\Gamma(t))=\overline{\iota}(\Gamma(t^-))$ by Fact~B, it is enough to prove that $\Gamma(t')$ approaches $\Gamma(t^-)$ in the space of paths as $t'$ approaches $t$ from below. This is the case, because the geodesic straightening of $\Gamma(t^-)$ again goes through infinitely many singularities $p_1\ p_2,\dots$ (\emph{not} including~$p$), and agrees with that of $\Gamma(t')$ up to any given $p_N$ provided $t-t'>0$ is small enough.

In the case $t'>t$, define similarly $\Gamma(t^+)$, the path obtained from $\Gamma(t)$ by tacking on a leaf of $\lambda^\mp$ making an angle $\pi$ with $\Gamma(t)$ \emph{above} $p$. This time, $p$ is the final turn of $\Gamma(t^+)$. Since $\overline{\iota}(\Gamma(t))=\overline{\iota}(\Gamma(t^+))$ by Fact~B, it is enough to prove that $\Gamma(t')$ approaches $\Gamma(t^+)$ as $t'$ approaches $t$ from above. This holds true because the next turn of $\Gamma(t')$ after $p$ lies arbitrarily far out in the horizontal direction if $t'-t>0$ is small enough.

If $t=0$, the argument is similar to the case $t'>t$ just treated. If $t=+\infty$, just observe that for very large $t'$, the geodesic straightening of $\Gamma(t')$ starts with an arbitrarily long, nearly vertical segment,
so again $\Gamma(t')$ approaches $\Gamma(+\infty)=\ell_i$ in the space of paths as $t'\rightarrow +\infty$.

Next, the identity $J_i(0)=J_i(+\infty)=\overline{\iota}(\Omega)$ is clear by Fact~B. Finally, we check the injectivity properties of $J_i$. Note that, after straightening, $\Gamma(t)$ and the other $t$-hook $\Gamma'(t)$ along $\ell_i$ (with the same initial segment) start into different quadrants: $Q(\Gamma(t))\neq Q(\Gamma'(t))$. Therefore, Proposition~\ref{prop:shootoff} applies, particularly the uniqueness of~$t$. 
\end{proof}

\subsection{Subdivision of quadrants}
In a given quadrant at $\Omega$, bounded by rays $\ell_i$ and $\ell_{i+1}$, the ruling singularities $\{p^i_s\}_{s\in\mathbb{Z}}$ form a naturally ordered $\mathbb{Z}$-sequence, with vertical and horizontal coordinates varying monotonically in opposite directions. To see this, one may for example consider for every $t>0$ the initial length-$t$ segment of the leaf $\ell_i$, and push this segment in the direction of $\ell_{i+1}$ until it bumps into a singularity. The area swept out is a rectangle, and the union of all these rectangles for all $t>0$ forms a ``staircase'', the $p^i_s$ being the turning points at the back of each stair. See Figure~\ref{staircase}. 
The indexing convention is such that the path $[\Omega, p^i_s]$ converges to $\ell_{i+1}$ (resp.\ $\ell_i$) as $s\rightarrow +\infty$ (resp.\ $s\rightarrow -\infty$). This defines $s$ only up to an additive shift depending on $i$, but we make no attempt to harmonize these shifts (i.e.\ any choice will do).
\begin{figure}[h!]
\centering
\labellist
\small\hair 2pt
\pinlabel $\Omega$ [c] at -2 2
\pinlabel $\ell_i$ [c] at -1 60
\pinlabel $\ell_{i+1}$ [c] at 110 2
\pinlabel $p^i_{s}$ [c] at 24 48
\pinlabel $p^i_{s+1}$ [c] at 49 26
\pinlabel $p^i_{s+2}$ [c] at 74 12
\pinlabel {$\mathcal{I}^i_{s+1}$} [c] at 46 13
\pinlabel {$\mathcal{I}^i_s$} [c] at 28 32
\pinlabel {$\mathcal{I}^i$} [c] at 72 58
\endlabellist
\includegraphics[width=7cm]{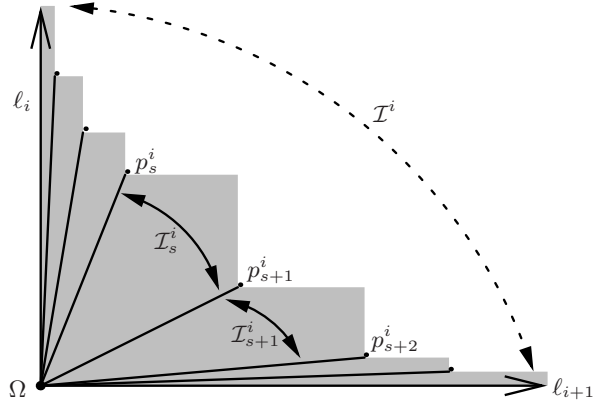}
\caption{A staircase in the $i$-th quadrant, with ruling segments $[\Omega, p_s]$. We mark the directions for initial segments of paths in the interval $\mathcal{I}^i_s$ of $\mathbb{S}$, as well as $\mathcal{I}^i=\bigcup_{s\in\mathbb{Z}}\mathcal{I}^i_s$.}
\label{staircase}
\end{figure} 
\begin{definition}
Inside the circle $\mathbb{S}$ of geodesic paths of $\overline{\Sigma}$ issued from $\Omega$, we let:
\begin{itemize}
 \item $\mathcal{I}^i$ be the closed interval of paths whose initial segment falls between the leaves $\ell_i$ and $\ell_{i+1}$;
 \item $\mathcal{I}^i_s\subset \mathcal{I}^i$ be the closed interval of paths whose initial segment falls between the ruling segments $[\Omega, p^i_s]$ and $[\Omega, p^i_{s+1}]$.
\end{itemize}
\end{definition}
\begin{remark} \label{boustro}
An important feature is that the order on the $\mathcal{I}^i_s$ induced by the cyclic order on $\mathbb{S}$ is the lexicographic order on pairs $(i,s)$. In particular, fixing $i\in\mathbb{Z}$,
\begin{itemize}
\item $\Gamma_i^-(t)\in \mathcal{I}^{i-1}_s$ for an index $s$ that is a \emph{nondecreasing} function of $t$;
\item $\Gamma_i^+(t)\in \mathcal{I}^{i}_s$ for an index $s$ that is a \emph{nonincreasing} function of $t$.
\end{itemize}
\end{remark}
Recall  the maps $J_i=\overline{\iota}\circ \Gamma_i^\pm :\overline{\mathbb{R}^+}\rightarrow \mathbb{P}^1\mathbb{C}$ from Proposition \ref{prop:jordan1}.
\begin{proposition} \label{prop:jordan2}
For $(i,t)$ and $(i', t')$ distinct elements of $\mathbb{Z} \times \overline{\mathbb{R}^+}$, the relationship $J_i(t)=J_{i'}(t')$ holds if and only if:
\begin{itemize}
 \item $\{t,t'\}\subset \{0,+\infty\}$, or
 \item $|i-i'|=1$ and the quadrant between $\ell_i$ and $\ell_{i'}$ contains a ruling singularity $p$ at coordinates $(t,t')$.
\end{itemize}
\end{proposition}
\begin{proof}
For the ``if'' direction, the first case is again the characterization of the $\overline{\iota}$-fiber of $\Omega$ by Fact~B. The second case follows similarly from the characterization of the $\overline{\iota}$-fiber of $p$ (more precisely, of the path represented by the ruling segment $[\Omega, p]$).

For the ``only if'' direction, suppose first that $J_i(t)=J_{i'}(t')$ is the point $\overline{\iota}(\Omega)$: by Fact~B, the corresponding hooks are degenerated to full leaves of $\lambda^\pm$, i.e.\ $t$ and $t'$ belong to $\{0,+\infty\}$. If $J_i(t)=J_{i'}(t')$ is some other point $P$, then we can apply the second and third cases of Proposition~\ref{prop:shootoff}: the fiber $\mathcal{F}=\overline{\iota}^{-1}(P)$ contains at most two (opposite) $t$-hooks along any leaf $\ell_i$, and this happens either for one value of $i$ (in which case $P$ belongs only to the Jordan curve $J_i(\overline{\mathbb{R}^+})$) or for two consecutive values of $i$ (in which case $P$ belongs to the two curves). The latter case arises precisely when $\mathcal{F}$ contains a ruling segment, and $(t,t')$ are then its coordinates.
\end{proof}

\subsection{Proof of Theorem~\ref{thm:jordan}}

Propositions \ref{prop:jordan1} and \ref{prop:jordan2} show that two Jordan curves $J_i(\overline{\mathbb{R}^+})\subset \mathbb{P}^1\mathbb{C}$ intersect only at $\infty$ and at a \emph{discrete} subset of $\mathbb{C}$, because ruling singularities do not accumulate. More precisely, $J_i(\overline{\mathbb{R}^+})\cup J_{i+1}(\overline{\mathbb{R}^+}) \smallsetminus \{\infty\}$ is the boundary of the union of an infinite chain of disks $(\delta^i_s)_{s\in\mathbb{Z}}$, each disk sharing just one boundary point with the next.

To finish proving Theorem~\ref{thm:jordan}, it remains to check that the Jordan curves $J_i(\overline{\mathbb{R}^+})$ never \emph{cross} each other (i.e.\ they bound nested disks $D_i$), and that the Cannon-Thurston map $\overline{\iota}$ fills the string of disks $D_i\smallsetminus D_{i+1}$ in linear order.

First, define $D'_i$ as the image under $\overline{\iota}$ of all paths whose initial segment lies clockwise from $\ell_i$:
$$D'_i:=\overline{\iota}\left (\overline{\bigcup_{j\geq i} \mathcal{I}^{j}} \right).$$

\begin{proposition} \label{fillup}
The set $D'_i$ is the closure of one complementary component of the Jordan curve $J_i(\overline{\mathbb{R}^+})$ in the sphere $\mathbb{P}^1\mathbb{C}$.
\end{proposition}
\begin{proof}
First, $D'_i$ is closed, as it is the image of a compact interval under a continuous map. 
Therefore $D'_i\smallsetminus J_i(\overline{\mathbb{R}^+})$ is closed in $ \mathbb{P}^1\mathbb{C}\smallsetminus J_i(\overline{\mathbb{R}^+})$; let us prove that it is also open.
Let $\gamma$ be a nontrivial path whose initial segment lies clockwise from $\ell_i$, and is not itself a leaf $\ell_{j}$.
 
Suppose $D'$ does not contain a neighborhood of $\overline{\iota}(\gamma)$. By surjectivity of the Cannon-Thurston map $\overline{\iota}$, this means some paths whose initial segment lies \emph{counterclockwise} from $\ell_i$ are mapped by $\overline{\iota}$ arbitrarily close to $\overline{\iota}(\gamma)$; taking limits it implies, by continuity of $\overline{\iota}$, that the $\overline{\iota}$-fiber $\mathcal{F}$ of $\gamma$ contains paths whose initial segment lies counterclockwise from $\ell_i$ (this is a compactness argument: the limiting paths $\ell_i$ and $\Omega$ of $\mathbb{S}\smallsetminus \bigcup_{j\geq i} \mathcal{I}^j$ are already known to belong to a different fiber).
 
The fiber $\mathcal{F}$ contains paths belonging to (the interiors of) $k$ consecutive quadrants for some $k\in\{1,2,3\}$, by Proposition~\ref{prop:shootoff}. The discussion above implies $k\geq 2$. The case $k=2$ means, by Proposition~\ref{prop:shootoff}, that some element of $\mathcal{F}$ is (the geodesic straightening of) a $t$-hook along $\ell_i$, hence $\overline{\iota}(\gamma)$ belongs to the Jordan curve $J_i(\overline{\mathbb{R}^+})$. By Proposition~\ref{prop:jordan2}, the case $k=3$ means that $\overline{\iota}(\gamma)$ belongs to the intersection of two Jordan curves: $J_i(\overline{\mathbb{R}^+}) \cap J_{i+1}(\overline{\mathbb{R}^+})$ or $J_i(\overline{\mathbb{R}^+}) \cap J_{i-1}(\overline{\mathbb{R}^+})$. In any case, $\overline{\iota}(\gamma)\in J_i(\overline{\mathbb{R}^+})$. This proves openness of $D'_i\smallsetminus J_i(\overline{\mathbb{R}^+})$ in $\mathbb{P}^1\mathbb{C} \smallsetminus J_i(\overline{\mathbb{R}^+})$.
 
To see that $D'$ is equal to only one side of  $J_i(\overline{\mathbb{R}^+})$, just remark that there are $\overline{\iota}$-fibers all of whose elements start off counterclockwise from $\ell_i$ (a $\overline{\iota}$-fiber, other than that of $\Omega$, occupies at most 3 consecutive quadrants by Proposition~\ref{prop:shootoff}).
\end{proof}

The proposition above implies that $D'_i$ equals $D_i$, the disk bounded by the $i$-th Jordan curve $J_i(\overline{\mathbb{R}^+})$. The inclusion $D_i\supset D_{i+1}$ is immediate from the definition of~$D'_i$. 
Therefore the closure of $D_i\smallsetminus D_{i+1}$ is the union of a $\mathbb{Z}$-sequence of disks $(\delta^i_s)_{s\in\mathbb{Z}}$, each having only one (boundary) point $\overline{\iota}(p^i_s)$ in common with the previous one. 

\begin{proposition} \label{fillup2}
For any $(i,s)\in\mathbb{Z}^2$, we have $\overline{\iota}(\mathcal{I}^i_s)=\delta^i_s$. 
\end{proposition}
\begin{proof}
It follows from Propositions \ref{prop:jordan2} and \ref{fillup} 
that the gate $\delta^i_s\cap \delta^i_{s+1}$ is the image under $\overline{\iota}$ of the $s$-th ruling segment $[\Omega, p^i_s]$ in the $i$-th quadrant.

Writing $p^i_s=p_s=(x_s, y_s)$ the natural coordinates in the quadrant ($y_s$ is along~$\ell_i$, and $x_s$ along $\ell_{i+1}$), the boundary of $\delta^i_s$ is a Jordan curve that can be broken up into two arcs:
$$\begin{array}{rrcl} &
\overline{\iota}(\Gamma_i^+([y_{s+1}, y_s])) &= & J_i([y_{s+1}, y_s]), \\  \text{and}
 &\overline{\iota}(\Gamma_{i+1}^-([x_s, x_{s+1}]))&=&J_{i+1}([x_s, x_{s+1}]).\end{array}$$
The strategy is now similar to the proof of Proposition~\ref{fillup}, replacing $\overline{\bigcup_{j\geq i} \mathcal{I}^j}$ with~$\mathcal{I}^i_s$. The set $\overline{\iota}(\mathcal{I}^i_s) \smallsetminus \partial \delta^i_s$ is closed in $\mathbb{P}^1\mathbb{C}\smallsetminus \partial \delta^i_s$ because $\mathcal{I}^i_s$ is compact; let us prove that it is also open. Let $\gamma\in \mathcal{I}^i_s$ be a path. By surjectivity and continuity of $\overline{\iota}$, it is enough to prove that if the $\overline{\iota}$-fiber $\mathcal{F}$ of $\gamma$ does not contain (the geodesic straightening of) a $t$-hook (i.e.\ $\overline{\iota}(\gamma) \notin \partial \delta_s^i$), then $\mathcal{F}$ is contained in $\mathcal{I}^i_s$. Contrapositively, prove that if $\mathcal{F}$ intersects two distinct subintervals $\mathcal{I}^i_s$ and $\mathcal{I}^{i'}_{s'}$, then $\mathcal{F}$ intersects two distinct intervals $\mathcal{I}^i$ and~$\mathcal{I}^{i''}$. 

If $\mathcal{F}$ is a singleton, then there is nothing to prove. If $\mathcal{F}$ consists of two elements, then these are paths $\gamma$ and $\gamma'$ asymptotic to the two ends of a foliation leaf $\ell$ (horizontal, say), as in Figure~\ref{merge} above: the two geodesics coincide up to a singularity~$p$, then diverge from each other, forming together with $\ell$ the boundary of a triangle containing no singularity. If $\gamma'\notin \mathcal{I}_s^i$, then the initial segments of $\gamma$ and $\gamma'$ are distinct, meaning that $p=\Omega$. There are several possible situations (Figure~\ref{block}):
\begin{figure}[h!]
\centering
\labellist
\small\hair 2pt
\pinlabel $\Omega$ [c] at 22 2
\pinlabel $p_{s+1}$ [c] at 61 16
\pinlabel $p_s$ [c] at 33 34
\pinlabel $\ell$ [c] at 50 45
\pinlabel $\gamma$ [c] at 73 33
\pinlabel $\gamma'$ [c] at 18 40
\pinlabel $y$ [c] at 29 50
\pinlabel $\ell_i$ [c] at 24 58
\pinlabel $\ell_{i+1}$ [c] at 96 2
\pinlabel $\Omega$ [c] at 122 2
\pinlabel $p_{s+1}$ [c] at 160 15
\pinlabel $p_s$ [c] at 135 39
\pinlabel $\ell$ [c] at 162 33
\pinlabel $\gamma$ [c] at 180 24
\pinlabel $\gamma'$ [c] at 103 24
\pinlabel $y$ [c] at 123 33
\pinlabel $\ell_i$ [c] at 128 49
\pinlabel $\ell_{i+1}$ [c] at 184 5
\endlabellist
\includegraphics[width=12.5cm]{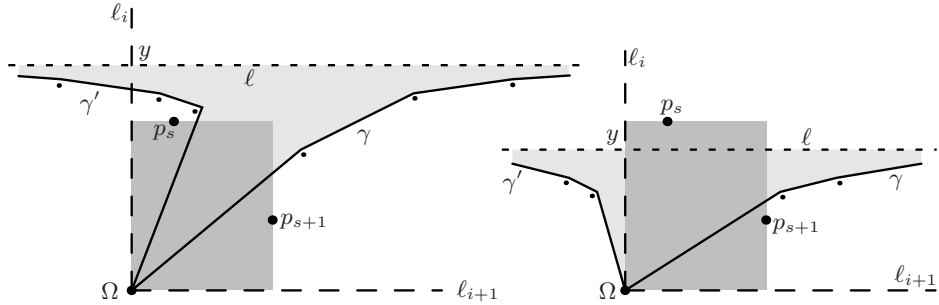}
\caption{Case $|\mathcal{F}|=2$. The shaded regions contain no singularity.}
\label{block}
\end{figure} 
\begin{itemize}
 \item If the height $y$ of the horizontal line $\ell$ is larger than $y_s$, then the paths $\gamma$ and $\gamma'$ both lie in $\mathcal{I}^i_s$. Indeed, the singularity $p_s$ prevents the geodesic straightening $\gamma'$ from having an initial segment with slope any larger than~$\frac{y_s}{x_s}$. (This will be the actual value if $y-y_s$ is small enough, but otherwise another singularity can force the slope to be even lower, as in the left panel of Figure~\ref{block}).
\item If $y\in (y_{s+1},y_s)$, then the absence of singularity in the rectangle $[0,x_{s+1}]\times [0,y_s]$ causes the initial segment of $\gamma'$ to lie in a another quadrant: $\mathcal{F}$ intersects a second interval $\mathcal{I}^{i\pm 1}$ in addition to $\mathcal{I}^i$. See the right panel of Figure \ref{block}. (In particular, $\gamma$ and $\gamma'$ are $t$-hooks, with $t=y$, and $\overline{\iota}(\gamma)\in \partial \delta^i_s$.)
\item The case $y<y_{s+1}$ is ruled out, as neither $\gamma$ nor $\gamma'$ would belong to the subinterval $\mathcal{I}^i_s$. This finishes the case $|\mathcal{F}|=2$.
\end{itemize}

Finally we discuss the case that the fiber $\mathcal{F}$ is infinite. The elements of $\mathcal{F}$ are a certain path $\overline{\gamma}$ terminating at a (not necessarily ruling) singularity $p$, and all the paths $(\gamma_n)_{n\in\mathbb{Z}}$ obtained from $\overline{\gamma}$ by tacking on a vertical or horizontal leaf issued from $p$. Each of these paths belongs to some $\mathcal{I}^{\mathbf{i}(\gamma_n)}_{\mathbf{s}(\gamma_n)}$, for some integers $\mathbf{i}(\gamma_n)$ and $\mathbf{s}(\gamma_n)$ determined by the initial segment of (the geodesic straightening of) $\gamma_n$.
If $\overline{\gamma}\in \mathcal{F}$ is a ruling segment, we have shown (see Figure~\ref{merge-3} and Proposition \ref{prop:shootoff}) that $\{\mathbf{i}(\gamma_n)\}_{n\in\mathbb{Z}}$ consists of three consecutive integers and the proof is finished. If not, then $(\mathbf{i}, \mathbf{s})$ is a well-defined function on $\mathcal{F}=\{\overline{\gamma}\}\cup\{\gamma_n\}_{n\in\mathbb{Z}}$, and we can assume $(\mathbf{i}, \mathbf{s})(\overline{\gamma})\neq (\mathbf{i}, \mathbf{s})(\gamma_n)$ for some $n\in \mathbb{Z}$. Since $\overline{\gamma}$ and $\gamma_n$ differ only by the final leaf from $p$, the situation is analogous to the case $|\mathcal{F}|=2$ (Figure \ref{block}). The only difference is that the leaf $\ell$ terminates at a point $p$, but the argument is exactly the same (applied to $\overline{\gamma}, \gamma_n$ instead of $\gamma, \gamma'$): if the initial segments of the two paths fall in distinct subintervals ($\mathcal{I}^{\mathbf{i}(\gamma_n)}_{\mathbf{s}(\gamma_n)}\neq \mathcal{I}^{\mathbf{i}(\overline{\gamma})}_{\mathbf{s}(\overline{\gamma})}$), then they fall in completely different quadrants: $\mathbf{i}(\gamma_n) \neq \mathbf{i}(\overline{\gamma})$.
\end{proof}

Theorem~\ref{thm:jordan} is now proved: the boustrophedonic order of filling of the discs $\delta^i_s$ derives from the opposite directions of monotonicity in Remark~\ref{boustro}.
In particular, we can speak of the objects define in Section \ref{sec:mainresult}: the $i$-th furrow $\bigcup_{s\in\mathbb{Z}} \delta^i_s$; cross-furrow and in-furrow edges in $\partial \delta^i_s$; gates $\delta^i_s\cap \delta^i_{s\pm 1}$ and spikes.

\section{Proof of Theorem~\ref{thm:dico}} \label{sec:dico}
The proofs of the previous section have made it clear that the maximal singularity-free rectangles play an important role in the combinatorics of the Cannon-Thurston map~$\overline{\iota}$. Since these rectangles also govern the Agol triangulation by Theorem~\ref{thm:veering}, a result such as Theorem~\ref{thm:dico} should now come as no surprise. In this section we establish the detailed correspondence.

\subsection{The dictionary or Theorem \ref{thm:dico}.(1)}\label{sec:dico1}
More precisely, we now discuss and illustrate the various entries in the right column (Agol triangulation features) of the dictionary table of Theorem \ref{thm:dico}. To each such feature, we associate first a type of rectangle in $\overline{\Sigma}$, and then a feature of the Cannon-Thurston tessellation (left column of the dictionary table). These correspondences will be clearly bijective, proving Theorem \ref{thm:dico}.(1).

All figures in this section will obey the following convention:
\begin{itemize}
\item Left panel: the Cannon-Thurston tessellation.
\item Middle panel: the singular Euclidean surface $\overline{\Sigma}$.
\item Right panel: the Agol triangulation (more precisely, its vertex link).
\end{itemize}
Each of the 3 panels may consist of several diagrams showing different possibilities (orientation or color reversals, even/odd quadrants etc): these diagrams are in natural correspondence between the 3 panels.

Note that the left and right panel both live in the (same) complex plane, while the middle panel lives in the singular surface $\overline{\Sigma}$. Each diagram in the middle panel is considered up to a $180^\circ$ rotation, since $\Sigma$ is only a \emph{half}-translation surface. Up to this ambiguity, the figures are designed to systematically capture all possible cases.
Note also the following correspondence:

\begin{center}
\begin{tabular}{r|l}
 In $\overline{\Sigma}$ (middle panel) & In the link of $\Omega$ (other panels) \\ \hline
 Clockwise rotation & Rightwards motion \\
 Vertical direction & Top end \\
 Horizontal direction & Bottom end.
\end{tabular}
\end{center}

\subsection*{Vertices (Figure~\ref{vertices})}
Ruling segments $[\Omega, p]$ in $\overline{\Sigma}$ correspond to vertices of the Agol triangulation by Theorem~\ref{thm:veering}, and to vertices of the Cannon-Thurston tessellation by Proposition~\ref{prop:jordan2}.
\begin{figure}[h!]
\centering
\labellist
\small\hair 2pt
\pinlabel {\emph{Cannon-Thurston tessellation}} [c] at 0 -7
\pinlabel {\emph{Flat surface $\overline{\Sigma}$}} [c] at 125 -7
\pinlabel {\emph{Agol triangulation}} [c] at 250 -7
\pinlabel $\Omega$ [c] at 144 5
\pinlabel $\Omega$ [c] at 106 42
\endlabellist
\includegraphics[width=10cm]{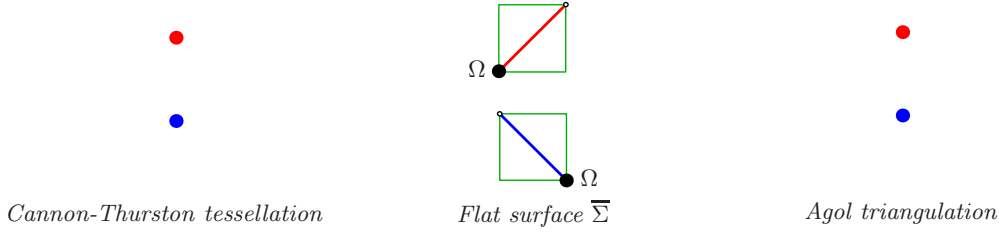}
\vspace{7pt}
\caption{In the middle panel, the green rectangle contains no singularity in its interior and $\Omega$ is the dark vertex; this convention is held up in all pictures of Section~\ref{sec:dico1}.}
\label{vertices}
\end{figure} 

\subsection*{Ladderpole edges (Figure~\ref{ladderpole})}
Recall that ladderpole edges in the Agol triangulation are by definition edges connecting two vertices of the same color. These two vertices necessarily come from two consecutive ruling edges of a given quadrant. The paths (of $\mathcal{I}^i_s$) inbetween these two ruling segments are mapped by $\overline{\iota}$ to exactly one 2-cell $\delta^i_s$ of the Cannon-Thurston tessellation, by Proposition~\ref{fillup2}.

Furrows $\bigcup_{s\in\mathbb{Z}} \delta^i_s$ correspond to full quadrants $\bigcup_{s\in\mathbb{Z}} \mathcal{I}^i_s=\mathcal{I}_s$ in $\overline{\Sigma}$, which correspond to full sequences of ruling singularities $(p^i_s)_{s\in\mathbb{Z}}$, and in turn to sequences of ladderpole edges, i.e.\ ladderpoles, in the Agol triangulation.
\begin{figure}[h!]
\centering
\labellist
\small\hair 2pt
\pinlabel $\Omega$ [c] at 184 19
\pinlabel $\Omega$ [c] at 184 57
\pinlabel $\Omega$ [c] at 291 19
\pinlabel $\Omega$ [c] at 291 57
\pinlabel {\emph{2-cells}} [c] at 77 -7
\pinlabel {\emph{Triangles, $\Omega$ in corner}} [c] at 235 -7
\pinlabel {\emph{Ladderpole edges}} [c] at 345 -7
\endlabellist
\includegraphics[width=11.5cm]{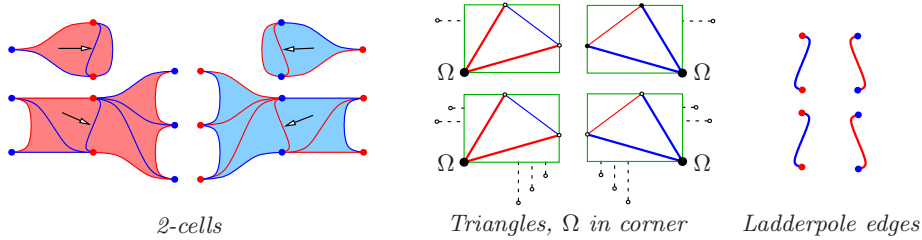}
\vspace{7pt}
\caption{In the middle panel, we have also indicated ruling singularities in the two quadrants adjacent to the one containing the singularity-free rectangle. These singularities correspond to spikes in the left panel. In the left panel, each arc indicated by an arrow is always the ladderpole edge seen in the corresponding diagram of the right panel.}
\label{ladderpole}
\end{figure} 

\subsection*{Nonhinge triangles (Figure~\ref{nonhinge})}
A triangle in the Agol triangulation corresponds by definition to a maximal singularity-free rectangle in $\overline{\Sigma}$ containing exactly one singularity in every edge, one of these singularities being $\Omega$. The triangle is non-hinge exactly when the two diagonals of the quadrilateral spanned by the four singularities have slopes of the same sign. Up to symmetry, we may then assume that $\Omega$ is the bottom vertex, and that the other three vertices $T, R, L$ (Top, Right, Left) satisfy: $x_L<x_T<0<x_R$ and $0<y_R<y_L<y_T$ where $(x_p, y_p)$ are the coordinates of a point $p$. 

\begin{figure}[h!]
\centering
\labellist
\small\hair 2pt
\pinlabel $\Omega$ [c] at 135 43
\pinlabel $R$ [c] at 148 58
\pinlabel $T$ [c] at 127 80
\pinlabel $L$ [c] at 113 66
\pinlabel $\tau$ [c] at 256 66
\pinlabel {\emph{In-furrow edges}} [c] at 35 -7
\pinlabel {\emph{Nonhinge tetrahedra}} [c] at 155 -7
\pinlabel {\emph{Nonhinge triangles}} [c] at 290 -7
\endlabellist
\includegraphics[width=12.5cm]{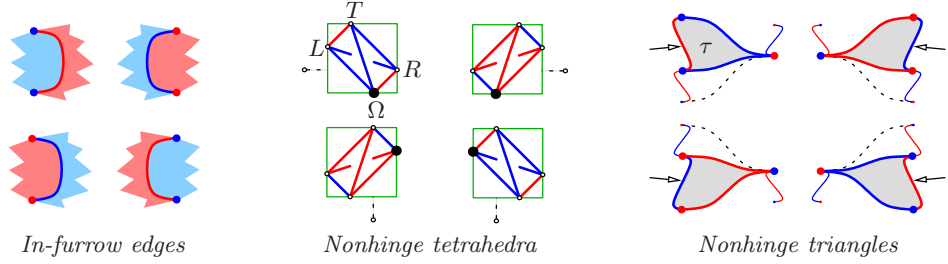}
\vspace{7pt}
\caption{In the right panel, the nonhinge triangle under consideration is shaded; we have also depicted the \emph{next} triangle across its base rung . The color of the base rung of \emph{that} second triangle (dashed) is determined by the position, in the middle panel, of the next ruling singularity after $L$ (also dashed), relative to the height $y_R$ of $R$. Finally, each edge indicated by an arrow in the right panel is isotopic to the edge seen in the corresponding diagram of the left panel.}
\label{nonhinge}
\end{figure} 

Equivalently, $L, T$ are two consecutive ruling singularities of the quadrant of $L$ and no ruling singularity in the quadrant of $R$ has a vertical coordinate in $[y_L, y_T]$. By Proposition~\ref{prop:jordan2}, this means that the Jordan curve $J_i(\overline{\mathbb{R}^+})$, associated to $t$-hooks along the vertical leaf issued from~$\Omega$, does not intersect $J_{i+1}(\overline{\mathbb{R}^+})$ between $J_i(y_L)$ and $J_i(y_T)$; the 
corresponding cell of the Cannon-Thurston tessellation thus has no spike on this side, but instead an in-furrow edge (vertical edge). 
\begin{remark} \label{rem:arrow1} This edge ``$LT$'' is topologically the same as an edge of the initial nonhinge triangle $\tau$ of the Agol triangulation: namely the edge connecting the tip of $\tau$ to the tip of the next triangle across the base rung of $\tau$. This edge is indicated by an arrow in Figure \ref{nonhinge}.
\end{remark}

\subsection*{Hinge triangles (Figure~\ref{hinge})}
Similarly, 
\begin{figure}[h!]
\centering
\labellist
\small\hair 2pt
\pinlabel $\Omega$ [c] at 134 43
\pinlabel $R$ [c] at 147 66
\pinlabel $T$ [c] at 126 80
\pinlabel $L$ [c] at 112 58
\pinlabel {\emph{Cross-furrow edges}} [c] at 37 -7
\pinlabel {\emph{Hinge tetrahedra}} [c] at 158 -7
\pinlabel {\emph{Hinge triangles}} [c] at 302 -7
\endlabellist
\includegraphics[width=12.5cm]{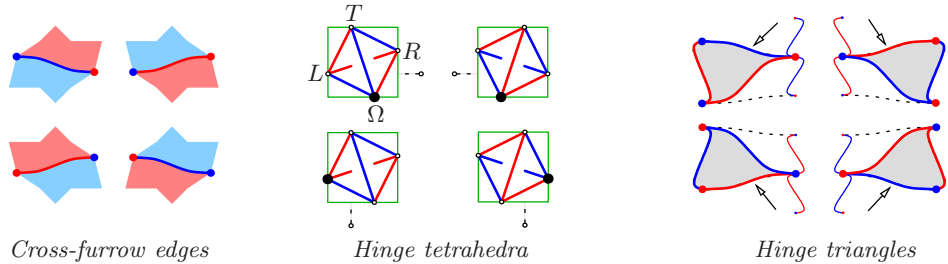}
\vspace{7pt}
\caption{In the right panel, the hinge triangle under consideration is shaded; we have also depicted the \emph{next} triangle across its base rung . The color of the base rung of \emph{that} second triangle (dashed) is determined by the position, in the middle panel, of the next ruling singularity after $R$ (also dashed), relative to the height $y_L$ of $L$.
Finally, each edge indicated by an arrow in the right panel is isotopic to the edge seen in the corresponding diagram of the left panel. 
}
\label{hinge}
\end{figure}
a hinge triangle in the Agol triangulation corresponds in $\overline{\Sigma}$ (up to symmetries) to a maximal singularity-free rectangle with vertices $\Omega, L, T, R$ such that $x_L<x_T<0<x_R$ and $0<y_L<y_R<y_T$. Equivalently, $R$ is the \emph{highest} of the ruling singularities in its quadrant whose vertical coordinate lies in $[y_L, y_T]$. By Proposition~\ref{prop:jordan2}, this means the separator $J_i(\overline{\mathbb{R}^+})$, followed downwards from the gate $J_i(y_T)$, has its first spike at $J_i(y_R)$. The arc $J_i([y_R, y_T])$ is a cross-furrow edge. 
\begin{remark} \label{rem:arrow2} This edge ``$RT$'' is topologically the same as an edge of the initial nonhinge triangle $\tau$ of the Agol triangulation: namely the edge connecting the tip of $\tau$ to the tip of the next triangle across the base rung of $\tau$. This edge is indicated by an arrow in Figure \ref{hinge}.
\end{remark}

\subsection*{Rungs (Figure~\ref{rung})}
By definition, a rung in the Agol triangulation is an edge connecting two vertices of distinct colors. If we call $L, R$ the corresponding ruling singularities in $\overline{\Sigma}$ (belonging to adjacent quadrants), then we can assume up to symmetry that $y_L<y_R$. The existence of a singularity-free rectangle circumscribed to $\{\Omega, R, L\}$ means that if $L'$ denotes the next higher ruling singularity after $L$ (in the quadrant of $L$), then  $y_R \in [y_L, y_{L'}]$. This in turn means the arc $J_i([y_L, y_{L'}])$, which connects the two gates of one cell of the Cannon-Thurston tessellation, has a spike at $J_i(y_R)$.

\begin{figure}[h!]
\centering
\labellist
\small\hair 2pt
\pinlabel $\Omega$ [c] at 151 37
\pinlabel $R$ [c] at 166 66
\pinlabel $L$ [c] at 135 54
\pinlabel {\emph{Spikes}} [c] at 39 -3
\pinlabel {\emph{Triangles, $\Omega$ in edge}} [c] at 165 -3
\pinlabel {\emph{Rungs}} [c] at 290 -3
\endlabellist
\includegraphics[width=12.5cm]{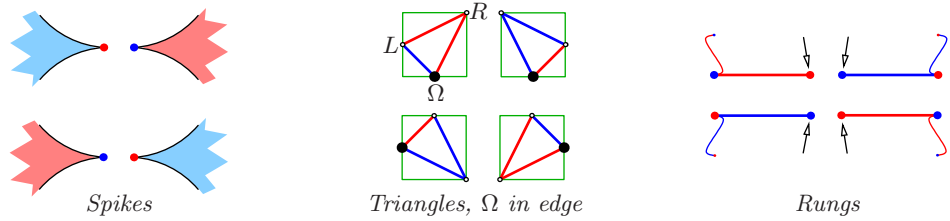}
\vspace{3pt}
\caption{In the right panel, each vertex indicated by an arrow is the vertex of the spike shown in the correponding diagram of the left panel. Also, in the right panel, we have drawn next to each rung a ladderpole edge (thinner) which is the one corresponding to the 2-cell in the left panel to which the spike belongs.}
\label{rung}
\end{figure} 

All correspondences above are clearly bijective. This proves the first, ``dictionary'' part of Theorem \ref{thm:dico}.

\subsection{The recipe book or Theorem \ref{thm:dico}.(2)}
Remarks \ref{rem:arrow1} and \ref{rem:arrow2} above, and the bijectivity of the dictionary, imply that the edges of the Cannon-Thurston tessellation are obtained from the Agol triangulation by drawing an edge between the tip of each triangle and the tip of the next triangle across its basis rung. This gives the recipe in the first direction of Theorem \ref{thm:dico}.(2).

Call $\mathcal{L}_i$ the path of tip-to-tip edges thus formed inside the $i$-th ladder. The fact that the second recipe in Theorem \ref{thm:dico}.(2) returns the original Agol triangulation can be seen purely at the level of the tessellation themselves, based on the fact that $\mathcal{L}_i$ visits all vertices of the $i$-th ladder exactly once, and that between any two consecutive vertices on one ladderpole, $\mathcal{L}_i$ visits a (possibly empty) sequence of consecutive vertices on the other ladderpole. Theorem \ref{thm:dico} is proved.

\section{Illustrations}
We finish with some extra illustrations to make the combinatorics of the filling curve $\overline{\iota}$ and the Cannon-Thurston tessellation more concrete. We will make no effort to draw its fractal-looking edges realistically: the information is still purely combinatorial.

\subsection{Global illustration}
In Figure~\ref{full-ladders} we show a relatively large sample of the Cannon-Thurston tessellation (left) and of the Agol triangulation (right), for the same underlying combinatorics. The picture shows in combination many of the local features described and illustrated in Section \ref{sec:dico1}.

\begin{figure}[h!]
\centering
\labellist

\small\hair 2pt
\pinlabel {\emph{Cannon-Thurston tessellation}} [c] at 60 -7
\pinlabel {\emph{Singular surface $\overline{\Sigma}$}} [c] at 185 -7
\pinlabel {\emph{Agol triangulation}} [c] at 300 -7
\endlabellist
\includegraphics[width=12.5cm]{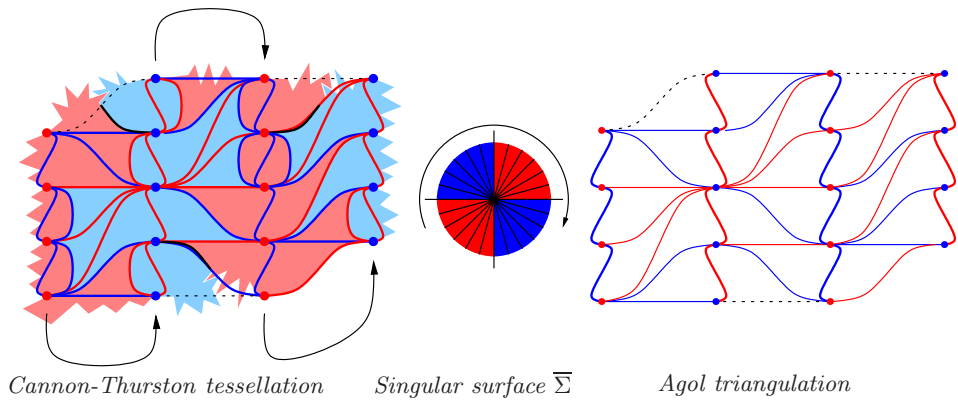} \vspace{7pt}
\caption{Corresponding chunks of the two plane tessellations (the combinatorics in $\overline{\Sigma}$ are not shown).}
\label{full-ladders}
\end{figure} 

In the left panel, only the solid colored areas contain the information on the Cannon-Thurston tessellation. However, we overlaid them with the 1-skeleton of the Agol triangulation to  help visualize the correspondences. Since all Cannon-Thurston edges are isotopic to Agol edges, each of the former receives a natural color (blue or red) even though it always just separates a blue area from a red one. Also, each ladderpole edge in the Agol triangulation is isotopic to 0, 1, or 2 in-furrow edges of the Cannon-Thurston tessellation: in the resulting overlay (left panel) we draw all 1,2 or 3 edges separately but with the same color. This was also the convention in the left panel of Figure \ref{ladderpole}.

In short, this synthesis is useful to illustrate Theorem \ref{thm:dico} and the interplay between the various lines of the ``dictionary''. However, it has two drawbacks:
\begin{itemize}
 \item We cannot draw the corresponding features in the flat surface $\overline{\Sigma}$, as the figure would become too crowded.
 \item In order to show exactly the same information in the two diagrams, we must leave some details ambiguous near the outer boundaries. 
\end{itemize}
For example, forcing a \emph{color} upon the top-left dotted edge in the Agol-triangulation panel would force an \emph{endpoint} (and a color) upon the top-left, unfinished edge of the Cannon-Thurston panel (interrupted here on a dotted arc). Conversely, choosing specific numbers of spikes for the outer cells of the Cannon-Thurston tessellation (presently truncated in a ragged style) would force some partial information upon the adjacent ladders (not pictured) in the Agol triangulation.

The first drawback is unavoidable if we draw too large a portion of the tessellations, but the second is unavoidable as long as we draw only a bounded portion.

\subsection{Semi-local illustration}
Thus, we try to strike a middle ground in Figure~\ref{full-cells}. This figure shows two full cells of the Cannon-Thurston tessellation (left), along with the corresponding data in $\overline{\Sigma}$ (middle) that allows to count their spikes, and the corresponding chunk of the Agol triangulation (right). In the middle panel, the whole green area is singularity-free.

The numbers of spikes chosen are $0,1$ in the top cell (on its left and right sides respectively), and $2,3$ in the bottom cell. 
Up to varying these numbers, one can build the Cannon-Thurston tessellation entirely out of $2$-cells like the blue one shown, and red ones obtained by a horizontal reflection and an exchange of colors.

\begin{figure}[h!]
\centering
\labellist
\small\hair 2pt
\pinlabel $a$ [c] at 13 132
\pinlabel $c$ [c] at 31 134
\pinlabel $b$ [c] at 12 93
\pinlabel $d$ [c] at 32 93
\pinlabel $a$ [c] at 17 60
\pinlabel $c$ [c] at 39 63
\pinlabel $b$ [c] at 17 16
\pinlabel $d$ [c] at 38 16
\pinlabel $a$ [c] at 105 142
\pinlabel $c$ [c] at 94 134
\pinlabel $b$ [c] at 88 122
\pinlabel $d$ [c] at 79 112
\pinlabel $a$ [c] at 105 79
\pinlabel $c$ [c] at 95 73
\pinlabel $b$ [c] at 76 48
\pinlabel $d$ [c] at 69 38
\pinlabel $a$ [c] at 166 123
\pinlabel $c$ [c] at 191 136
\pinlabel $b$ [c] at 168 90
\pinlabel $d$ [c] at 193 95
\pinlabel $a$ [c] at 175 62
\pinlabel $c$ [c] at 195 65
\pinlabel $b$ [c] at 174 14
\pinlabel $d$ [c] at 194 15
\pinlabel $\Omega$ [c] at 114 28
\pinlabel $\Omega$ [c] at 114 103
\endlabellist
\includegraphics[width=12.5cm]{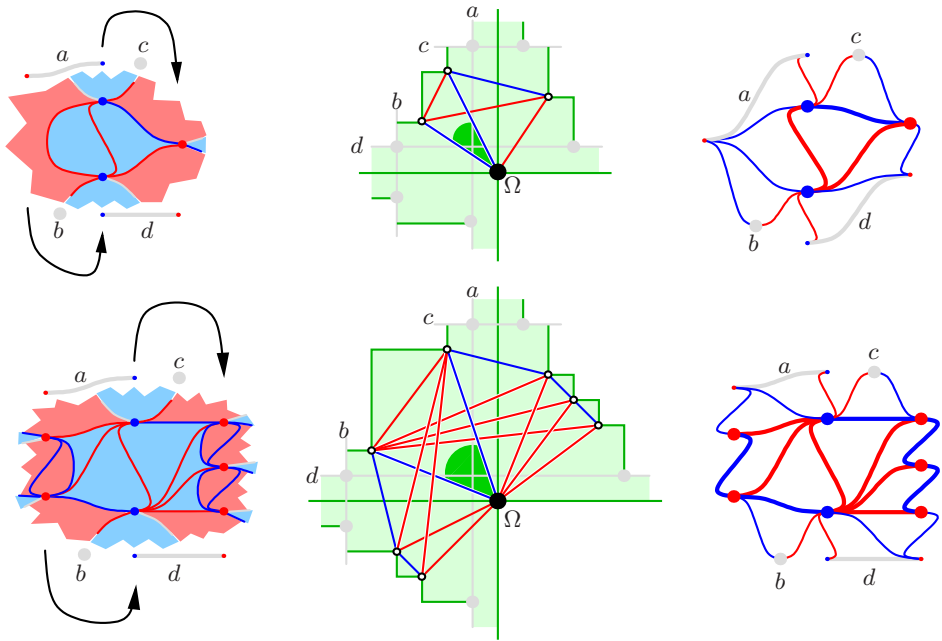}
\caption{The combinatorics associated to full cells (left panel) of the Cannon-Thurston tessellation: the first cell has $0$ spikes on one side and $1$ on the other; the second cell has $2$ and $3$. The corresponding subinterval $\mathcal{I}^i_s$ is shown as a green sector in the middle panel.}
\label{full-cells}
\end{figure} 

The portrayed cell in the left panel is the image, under the Cannon-Thurston map~$\overline{\iota}$, of a subinterval $\mathcal{I}^i_s$ in the \emph{upper left} quadrant in the middle panel. The corresponding ladderpole edge in the right panel is the thick red one (connecting blue vertices).

Moreover, the uncertainty about how the portrayed Cannon-Thurston cell is glued up to its neighbors gives rise to corresponding uncertainties in the other diagrams as well, of which we keep careful track. Namely, letters $a,b,c,d$ refer: 
\begin{itemize}                                                                                                                                                                                                                                                                \item (Fig.\ \ref{full-cells}, left) to the indeterminacy about the endpoint and color of an edge ($a,d$) or the color and position of a vertex ($b,c$);
\item (Fig.\ \ref{full-cells}, middle) to the indeterminacy about which way the tie breaks between the coordinates of two ruling singularities in adjacent quadrants (connected by a grey line);
\item (Fig.\ \ref{full-cells}, right) to the indeterminacy about the color of an edge ($a,d$) or the color and position of a vertex ($b,c$).
\end{itemize}

\subsection{Further subdivisions of the Cannon-Thurston map}

In Figure~\ref{detail014}, we show a schematic diagram of the order in which a cell $\delta^i_s=\overline{\iota}(\mathcal{I}^i_s)$ of the Cannon-Thurston tessellation is filled out. This involves two consecutive ruling singularities $p_s, p_{s+1}$, as well as a maximal sequence of $k\geq 2$ singularities such that any two consecutive of them, together with $p_s$ and $p_{s+1}$, span a maximal singularity-free rectangle with $[p_s, p_{s+1}]$ as the bottom left edge. We show the cases $k=2,3,6$.

Note that for each path $\gamma$ that terminates at a singularity, such as the paths numbered 0, 3, 8, 11 in the case $k=2$ (top), the curve $\overline{\iota}$ actually goes infinitely many times through $\overline{\iota}(\gamma)$ in a neighborhood of $\gamma$ (we just draw one passage for simplicity). Indeed, near $\overline{\iota}(\gamma)$, the trajectory of $\overline{\iota}$ is actually the image of the full boustrophedon under a M\"obius map, with $\overline{\iota}(\gamma)$ playing the role of the point at infinity. This also explains why so-called spikes do look ``spiky'', in the sense that they are pinched between two tangent circles, the M\"obius image of a pair of parallel lines.


\begin{figure}[h!]
\centering
\labellist
\pinlabel $\Omega$ [c] at 142 250
\pinlabel $p_s$ [c] at 157 278
\pinlabel $p_{s+1}$ [c] at 171 257
\pinlabel $\Omega$ [c] at 148 130
\pinlabel $p_s$ [c] at 163 158
\pinlabel $p_{s+1}$ [c] at 177 137
\pinlabel $\Omega$ [c] at 154 14
\pinlabel $p_s$ [c] at 169 42
\pinlabel $p_{s+1}$ [c] at 183 21
\endlabellist
\includegraphics[width=12.5cm]{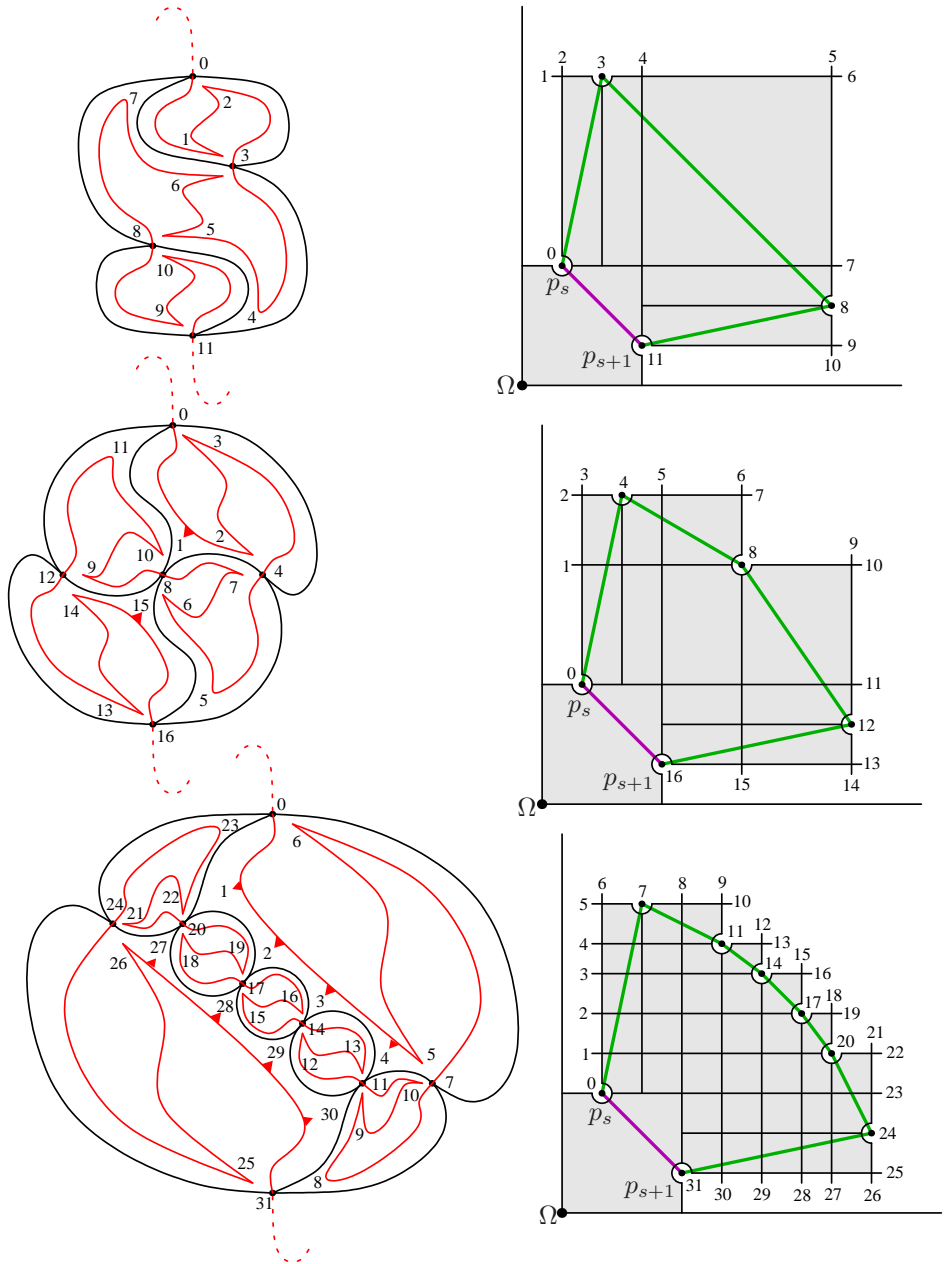}
\caption{Left: the filling curve $\overline{\iota}$, filling a disk $\delta^i_s$. Right: the paths issued from $\Omega$ and crossing the purple segment $[p_s, p_{s+1}]$ can be partitioned according to which green segment they cross next; $\overline{\iota}$ maps this partition to a partition of $\delta^i_s$ into smaller disks. A form of recursion then applies (we show only the first step).}
\label{detail014}
\end{figure}

\end{document}